\documentclass[11pt,a4paper,openany]{amsart}

\usepackage{amsmath,amsfonts,ams}
\usepackage{latexsym}
\usepackage{amssymb}
\usepackage{color}
\usepackage{epsfig}
\usepackage{subfigure}
\usepackage{mathrsfs}
\usepackage{amscd}
\usepackage{amsthm}
\usepackage{color}
\usepackage{cite}

\numberwithin{equation}{section}
\newtheorem{proposition}{Proposition}[section]

\newtheorem{lemma}{Lemma}[section]
\newtheorem{theorem}{Theorem}[section]
\newtheorem{corollary}{Corollary}[section]
\newtheorem{remark}{Remark}[section]
\newtheorem{conjecture}{Conjecture}[section]

\begin{document}
\title{ Linear adjoint restriction estimates for paraboloid}
\author{Changxing Miao}
\address{Institute of Applied Physics and Computational Mathematics, P. O. Box 8009, Beijing, China, 100088}
\email{miao\_changxing@iapcm.ac.cn}

\author{Junyong Zhang}
\address{Department of Mathematics, Beijing Institute of Technology, Beijing, China, 100081; Cardiff University, UK} \email{zhang\_junyong@bit.edu.cn; ZhangJ107@cardiff.ac.uk}

\author{Jiqiang. Zheng}
\address{Institute of Applied Physics and Computational Mathematics, P. O. Box 8009, Beijing, China, 100088}
\email{zhengjiqiang@gmail.com}

\begin{abstract}
We prove  a class of modified paraboloid restriction estimates with a loss of angular derivatives
for the full set of paraboloid restriction conjecture indices. This
result generalizes the paraboloid restriction estimate in radial case from [Shao, Rev. Mat. Iberoam. 25(2009), 1127-1168],
 as well as the result from [Miao et al. Proc. AMS 140(2012), 2091-2102].
 As an application, we
show a local smoothing estimate for a solution of the linear Schr\"odinger equation under the assumption that the initial datum has additional angular regularity.
\end{abstract}


\maketitle

\section{Introduction}

Let  $S$ be a non-empty smooth compact subset of the paraboloid,
$$\big\{~(\tau,\xi)\in\R\times \R^n: \tau=|\xi|^2~\big\},$$ where $n\geq1$.
 We denote by $d\sigma$ the pull-back of
the $n$-dimensional Lebesgue measure $d\xi$ under the
projection map $(\tau,\xi)\mapsto \xi$.
Let $f$ be a Schwartz function and define the inverse space-time Fourier transform of the
measure $fd\sigma$
\begin{align}\label{def}
(fd\sigma)^{\vee}(t,x)&=\int_{S} f(\tau, \xi)e^{2\pi
i(x\cdot \xi+t\tau)}d\sigma(\xi) \\
&=\int_{\R^n} f(|\xi|^2,
\xi)e^{2\pi i(x\cdot \xi+t|\xi|^2)}d\xi.\nonumber
\end{align}
The classical linear adjoint restriction estimate for the paraboloid reads
\begin{equation}\label{res}
\|(fd\sigma)^{\vee}\|_{L^q_{t,x}(\R\times\R^n)}\leq
C_{p,q,n,S}\|f\|_{L^p(S;d\sigma)},
\end{equation}
where $1\leq p,q\leq\infty$. The famous restriction
problem is to find the optimal range of $p$ and $q$ such that the estimate \eqref{res} holds.
It is known that the condition
\begin{equation}\label{q-p}
q>\frac{2(n+1)}{n} \quad\text{and}\quad \frac{n+2}q\leq \frac{n}{p'},
\end{equation}
is necessary  for
\eqref{res}, see \cite{Stein, Tao}. Here $p'$ denotes the conjugate exponent of  $p$.
 The adjoint restriction estimate conjecture on paraboloid reads as follows.
\begin{conjecture}\label{conjecture} The
inequality \eqref{res} holds true if and only if inequalities \eqref{q-p} are valid.
\end{conjecture}
There is a large amount of literature on this problem. For $n=1$, Conjecture \ref{conjecture} was proved by
Fefferman-Stein \cite{FS} for the non-endpoint case and by Zygmund \cite{Zygmund} for the endpoint case.
Conjecture \ref{conjecture} in  high dimension case becomes  much more difficult.
For $n\geq2$, Tomas \cite{Tomas} showed \eqref{res} for
$q>{2(n+2)}/n$, and Stein \cite{Stein1} fixed the limit case
$q={2(n+2)}/n$. Bourgain \cite{Bourgain} further proved estimate
\eqref{res} for $q>2(n+2)/n-\epsilon_n$ with some $\epsilon_n>0$; in
particular, $\epsilon_n=\frac2{15}$ when $n=2$. Further improvements were made by Moyua-Vargas-Vega \cite{MVV}
and Wolff \cite{Wolff}. Tao \cite{Tao3} used the bilinear argument to
show that estimate \eqref{res} holds true for $q>{2(n+3)}/{(n+1)}$ with $n\geq2$. This result
was improved by Bourgain-Guth \cite{BG} when $n\geq4$. This
conjecture is so difficult that it remains open up to now. For more details, we
refer the reader to \cite{BG, Tao, Tao2, Tao3, TVV, Wolff}.
\vspace{0.2cm}

On the other hand, the restriction conjecture becomes simpler (but not
trivial) when a test function has some angular regularity.  For
example,  Conjecture \ref{conjecture} is proved by Shao
\cite{Shao} when test functions are cylindrically symmetric and are
supported on a dyadic subset of the paraboloid in the form of
$$\Big\{(\tau,\xi)\in \R\times\R^n:\; M\leq |\xi|\leq2M,\;\;  \tau=|\xi|^2,\;\; M\in 2^{\Z}\Big\}.$$
Indeed, many famous  conjectures in harmonic analysis (such as Fourier
restriction estimates, Bochner-Riesz estimate  etc.)
have easier counterparts when the corresponding operators act on radial
functions. Let $\mathbb{S}^{n-1}$ denote the unit sphere  in $\R^n$ and $L^q_{\text{sph}} := L^{q}_\theta(\mathbb{S}^{n-1})$, the intermediate situation is to replace the $L^q(\R^n)$ by $L^q_{r^{n-1}dr}L^2_{\text{sph}}$ in \eqref{res}.
 This intermediate
case has been settled for  adjoint restriction estimates for a cone
by the authors of \cite{MZZ}. More precisely, if $S$ is a non-empty
smooth compact subset of the cone:
$$S=\big\{(\tau,\xi)\in\R\times \R^n:\;\; \tau=|\xi|\big\},$$
then for $q>{2n}/(n-1)$ and $(n+1)/q\leq (n-1)/p'$ we have
\begin{equation}\label{mres}
\|(fd\sigma)^{\vee}\|_{L^q_t(\R;L^q_{r^{n-1}dr}L^{2}_{\text{sph}})}\leq
C_{p,q,n,S}\|f\|_{L^p(S;d\sigma)}.
\end{equation}
The $L^2_{\text{sph}}$-norm allows us to use
spherical harmonic expanding, so the problem is converted to $L^q(\ell^2)$-bounds for sequences of operators $\{H_{k}\}$
where each $H_k$ is an operator acting on radial functions. The pioneering paper using such intermediate space is the Mockenhaupt
Diploma in which he proved weighted $L^p$ inequalities and
then sharp $L^p_{\mathrm{rad}}(L^2_\mathrm{sph})\to L^p_{\mathrm{rad}}(L^2_\mathrm{sph})$ estimates for the disc multiplier operator,
see either Mockenhaupt \cite{Moc} or C\'ordoba \cite{Cor}.
Sharp endpoint bounds for the disk multiplier were obtained by Carbery-Romera-Soria \cite{CRS}.
M\"uller-Seeger \cite{MS} established some sharp mixed spacetime $L^p_{\mathrm{rad}}(L^2_\mathrm{sph})$ estimates in order
to study a local smoothing of solutions for the linear wave equation. C\'ordoba-Latorre \cite{CL} revisited some classical conjecture including restriction estimate in harmonic analysis in this kind of mixed space-time. Gigante-Soria \cite{GS} studied
a related mixed norm problem for Schr\"odinger maximal operators.
Concerning the sphere restriction conjecture, Carli-Grafakos \cite{CG} also treated the same problem for spherically-symmetric functions and Cho-Guo-Lee\cite{CGL} showed
a restriction estimate for $q>2(n+1)/n$ and $s\geq(n+2)/q-n/2$
\begin{equation}\label{sres}
\left\|\int_{\mathbb{S}^{n}} e^{2\pi ix\cdot\xi} f(\xi)d\sigma(\xi)\right\|_{L^q(\R^{n+1})}\leq
C\|f\|_{H^s(\mathbb{S}^{n})}, \quad x\in \R^{n+1},
\end{equation}
where $d\sigma$ is the induced Lebesgue measure on $\mathbb{S}^{n}$ and $H^s(\mathbb{S}^{n})$ denote the $L^2$-Sobolev space of order $s$ on the sphere.
An advantage of the proof consists in a fact that inequality \eqref{sres} is based on $L^2$-spaces. The advantage of using the $L^2$-based Hilbert space also allows us to use effective
the $TT^*$ arguments to obtain Strichartz estimate with a wider range of admissible indexes by compensating with extra regularity in angular direction; see Sterbenz \cite{Sterbenz} for wave equation, Cho-Lee\cite{CL} for
general dispersive equations and the authors \cite{MZZ1} for wave equation with an inverse-square potential. Concerning other results in this direction,  Cho-Hwang-Kwon-Lee \cite{CHKL} studied profile
decompositions of fractional Schr\"odinger equations under the angular regularity assumption.\vspace{0.2cm}

In this paper, we prove that estimate \eqref{res} holds  for all $p,q$ in \eqref{q-p} by compensating with some loss of angular derivatives. Our strategy is to use a spherical harmonic expanding as well as localized restriction estimates. In contrast to the radial case, e.g. \cite{Shao, CG},
the main difficulty  comes from  the  asymptotic
behavior of the Bessel function $J_{\nu}(r)$
 when  $\nu\gg1$. It is worth to point out
that the method of treating cone restriction \cite{MZZ}
 is not valid since it can not be used to exploit the curvature property
 of paraboloid multiplier $e^{it|\xi|^2}$.  We note that the bilinear argument used in \cite{Shao}, which is in spirit of
Carleson-Sj\"olin argument or equivalently the $TT^*$ argument, can be used to deal with
the oscillation of the paraboloid multiplier. To use this argument, one needs to write the Bessel function $J_\nu(r)\sim c_\nu r^{-1/2}e^{ir}$ when $r\gg1$.
This expression works well for small $\nu$ (corresponding to the radial case) but
it seems complicate  to write the Bessel function in that form when $\nu\gg1$. Indeed, as in \cite{Zhang}, one can do this when $\nu^2\ll r$, but it will cause more loss of derivative for the case
 $\nu\lesssim r\lesssim \nu^2$,  since it is difficult to capture simultaneously
the oscillation and decay behavior of $J_{\nu}(r)$. Our new idea here is to establish a $L^4_{t,x}$-localized restriction estimate by directly analyzing
the kernel associated with the Bessel function. The key ingredient
is to explore the decay and oscillation property of $J_\nu(r)$ for $r\gg \nu$,  and resonant  property of paraboloid multiplier.
We also have to overcome low decay shortage of $J_{\nu}(r)$ (when $\nu\sim r\gg1$) by compensating a loss of angular regularity.
\vspace{0.2cm}

Before stating the main theorem, we introduce some notation. Incorporating the
angular regularity, we set the infinitesimal generators of the
rotations on Euclidean space:
\begin{equation*}
\Omega_{j,k}:=x_j\partial_k-x_k\partial_j
\end{equation*}
and define for $s\in\R$
\begin{equation*}
\Delta_\theta:=\sum_{j<k}\Omega_{j,k}^2,\quad
|\Omega|^s=(-\Delta_{\theta})^{\frac{s}2}.
\end{equation*}
Hence $\Delta_{\theta}$ is the Laplace-Beltrami operator on
$\mathbb{S}^{n-1}$. Define the Sobolev norm $\|\cdot\|_{H^{s,p}_{\text{sph}}(\R^n)}$ by setting
\begin{equation}
\|g\|^p_{H^{s,p}_{\text{sph}}(\R^n)}=\int_0^\infty\int_{\mathbb{S}^{n-1}}|(1-\Delta_{\theta})^{s/2}g(r\theta)|^p d\theta ~r^{n-1}dr.
\end{equation}
Given a constant $A$, we briefly write $A+\epsilon$ as $A_+$ or
$A-\epsilon$ as $A_-$ for $0<\epsilon\ll 1$. \vspace{0.2cm}

Our main result is the following one.
\begin{theorem}\label{thm} Let $n\geq2$. The following estimates hold for all Schwartz functions $f$

 $\bullet$ if $q_0=(2(n+1)/n)_+$ and $(n+2)/q_0=n/p_0'$, then
\begin{equation}\label{q0}
\|(fd\sigma)^{\vee}\|_{L^{q_0}_{t,x}(\R\times\R^n)}\leq
C_{p,q_0,n,S}\|f(|\xi|^2,\xi)\|_{H^{\sigma_0,p_0}_{\mathrm{sph}}(\R^n_{\xi})},
\end{equation}
where  $\sigma_0=(n-2)\big(\frac12-\frac1{q_0}\big)+\frac2{q_0}$;

$\bullet$  if $1\leq q,p\leq \infty$ satisfy \eqref{q-p}, then
\begin{equation}\label{main}
\|(fd\sigma)^{\vee}\|_{L^q_{t,x}(\R\times\R^n)}\leq
C_{p,q,n,S}\|(1+|\Omega|)^{s} f\|_{L^p(S;d\sigma)},
\end{equation} where $s=s(q,n)=\sigma_0\alpha$ and $0\leq\alpha\leq1$ satisfying
$1/q=\alpha/q_0+(1-\alpha)/q_1$.
Here  $q_1=q(n)_+$ with $q(n)=2+12/(4n+1-k)$ if $n+1\equiv k (\text{mod}~3), k=-1,0,1$ as in Bourgain-Guth\cite[Theorem 1]{BG}.
\end{theorem}

\begin{remark}
Estimate \eqref{main} is an interpolation consequence of \eqref{q0} and $L^p$-estimates in Bourgain-Guth\cite{BG}. Inequality \eqref{main} leads to the  linear adjoint restriction
estimate when $q\in (2(n+1)/n, q(n)]$ with some loss of angular derivatives.
\end{remark}

\begin{remark}
Since the sphere $\mathbb{S}^n=\{(\tau,\xi): |\tau|^2+|\xi|^2=1\}$ is closely related to the paraboloid in sense of Taylor expansion $\sqrt{1-\rho^2}=1-\frac12 \rho^2+O(\rho^4)$ near $\rho=0$,
it seems to be possible to show some modified version of \eqref{sres} with $H^{s,p}(\mathbb{S}^n)$-norm on right hand side.
\end{remark}

As an application of the modified restriction estimate, we show a
result on the local smoothing estimate for the Sch\"odinger equation for
initial data with  additional conditions angular regularity by  Rogers's
argument in \cite{Rogers}.
Our result here extend \cite [Theorem 1] {Rogers} from $q>2(n+3)/(n+1)$ to
$q>2(n+1)/n$ under the assumption that initial data has additional angular regularity. \vspace{0.2cm}

More precisely, we have the following local smoothing result.
\begin{corollary}\label{cor} Let $n\geq2$, $q>2(n+1)/n$ and $s$ be  as in Theorem \ref{thm}. Then
\begin{equation}\label{1.4}
\|e^{it\Delta}u_0\|_{L^q_{t,x}([0,1]\times\R^n)}\leq
C\big\|(1+|\Omega|)^su_0\big\|_{W^{\alpha,q}(\R^n)},
\end{equation}
where $\alpha>2n(1/2-1/q)-2/q$ and $W^{\alpha,q}(\R^n)$ is the Sobolev space.
\end{corollary}

This paper is organized as follows: In Section 2, we introduce notation and present some
basic  facts about spherical harmonics and Bessel functions. Furthermore, we use the stationary
phase argument to prove some properties of  Bessel functions.
Section 3 is devoted to the proof of Theorem \ref{thm}. In Section
4, we prove the key Proposition \ref{linear estimates}. We prove Corollary \ref{cor} in the final section.\vspace{0.2cm}

{\bf Acknowledgements:}\quad   The authors would like to express
their great gratitude to S. Shao for his helpful discussions. The authors were supported by the NSFC under
grants 11771041, 11831004 and H2020-MSCA-IF-2017(790623).\vspace{0.2cm}

\section{Preliminaries}

\subsection{Notation.}  We use $A\lesssim B$ to denote the
statement that $A\leq CB$ for some large constant $C$ which may vary
from line to line and depend on various parameters,  and similarly
employ $A\sim B$ to denote the statement that $A\lesssim B\lesssim
A$. We also use $A\ll B$ to denote the statement $A\leq C^{-1} B$.
If a constant $C$ depends on a special parameter other than the
above, we shall write it explicitly by subscripts. For instance,
$C_\epsilon$ should be understood as a positive constant not only
depending on $p, q, n$ and $S$, but also on $\epsilon$. Throughout
this paper, pairs of conjugate indices are written as $p, p'$, where
$\frac{1}p+\frac1{p'}=1$ with $1\leq p\leq\infty$. Let $R>0$ be a dyadic number, we define the dyadic annulus in $\R^n$ by
$$A_{R}:=\big\{~x\in\R^n: \;\; R/2\leq|x|\leq R~\big\},\quad
S_{R}:=[R/2, R].$$ For each $M\in 2^{\Z}$, we define ${\mathbb L}_M$ to be the class of
Schwartz functions supported on a dyadic subset of the paraboloid in
the form of
\begin{equation}\label{dya}
\big\{(\tau,\xi)\in\R\times\R^n: M\leq|\xi|\leq2M, \tau=|\xi|^2\big\}.\end{equation}
\vspace{0.12cm}
\subsection{Spherical harmonics expansions and Bessel function}
We recall an expansion formula with respect to the
spherical harmonics. Let
\begin{equation}\label{2.2}
\xi=\rho \omega \quad\text{and}\quad x=r\theta\quad\text{with}\quad
\omega,\theta\in\mathbb{S}^{n-1}.
\end{equation}
For every $g\in L^2(\R^n)$, we have the expansion formula
\begin{equation*}
g(
\xi)=\sum_{k=0}^{\infty}\sum_{\ell=1}^{d(k)}a_{k,\ell}(\rho)Y_{k,\ell}(\omega),
\end{equation*}
where
\begin{equation*}
\big\{Y_{k,1},\ldots, Y_{k,d(k)}\big\}
\end{equation*}
is the orthogonal basis of the spherical harmonics space of degree
$k$ on $\mathbb{S}^{n-1}$. This space is recorded by $\mathcal{H}^{k}$ and it has the
dimension
\begin{equation*}
d(k)=\frac{2k+n-2}{k}C^{k-1}_{n+k-3}\simeq \langle k\rangle^{n-2}.
\end{equation*}
It is clear that we have the orthogonal
decomposition of $L^2(\mathbb{S}^{n-1})$
\begin{equation*}
L^2(\mathbb{S}^{n-1})=\bigoplus_{k=0}^\infty \mathcal{H}^{k}.
\end{equation*}
It follows that
\begin{equation}\label{2.3}
\|g(\xi)\|_{L^2_\omega}=\|a_{k,\ell}(\rho)\|_{\ell^2_{k,\ell}}.
\end{equation}
Using the spherical harmonic expansion, as well as \cite{MNNO, Ster}, we define the action of $(1-\Delta_\omega)^{s/2}$ on  $g$ as follows
\begin{equation}\label{a2.3}
(1-\Delta_\omega)^{s/2} g=\sum_{k=0}^{\infty}\sum_{\ell=1}^{d(k)}(1+k(k+n-2))^{s/2}a_{k,\ell}(\rho)Y_{k,\ell}(\omega).
\end{equation}
Given $s,s'\geq0$ and $p,q\geq 1$, define
\begin{equation*}
\begin{split}
\|g\|_{H^{s,q}_{\rho}H^{s',p}_{\omega}}:=\big\|(1-\Delta)^{\frac
s2}\big((1-\Delta_\omega)^{\frac
{s'}2}g\big)\big\|_{L^{q}_{\mu(\rho)}(\R^+;L^{p}_{\omega}(\mathbb{S}^{n-1}))},
\end{split}
\end{equation*}
where $\mu(\rho)=\rho^{n-1}d\rho$.\vspace{0.2cm}

For our purpose, we need the inverse Fourier transform of
$a_{k,\ell}(\rho)Y_{k,\ell}(\omega)$. We recall the Bochner-Hecke formula, see\cite{Howe} and \cite[Theorem 3.10]{SW}
\begin{equation}\label{Fourier}
\check{g}(r\theta)=\sum_{k=0}^{\infty}\sum_{\ell=1}^{d(k)}2\pi
i^{k}Y_{k,\ell}(\theta)r^{-\frac{n-2}2}\int_0^\infty
J_{\nu(k)}(2\pi r\rho)a_{k,\ell}(\rho)\rho^{\frac
n2}d\rho.
\end{equation}
Here $\nu(k)=k+\frac{n-2}2$ and the Bessel function $J_{\nu}(r)$ of order $\nu$ is defined by
\begin{equation*}
J_{\nu}(r)=\frac{(r/2)^\nu}{\Gamma(\nu+\frac12)\Gamma(1/2)}\int_{-1}^{1}e^{isr}(1-s^2)^{(2\nu-1)/2}\mathrm{d
}s,
\end{equation*} where $\nu>-1/2$ and $r>0$.
It is easy to verify that there exists a constant $C$ independent of $\nu$ such that
\begin{equation}\label{rf}
|J_\nu(r)|\leq
\frac{Cr^\nu}{2^\nu\Gamma(\nu+\frac12)\Gamma(1/2)}\Big(1+\frac1{\nu+1/2}\Big).
\end{equation}
To investigate a behavior of asymptotic bound on $\nu$ and $r$, we
recall  the Schl\"afli integral representation \cite{Watson} of the
Bessel function: for $r\in\R^+$ and $\nu>-\frac12$
\begin{equation}\label{bessel}
\begin{split}
J_\nu(r)&=\frac1{2\pi}\int_{-\pi}^\pi
e^{ir\sin\theta-i\nu\theta}d\theta-\frac{\sin(\nu\pi)}{\pi}\int_0^\infty
e^{-(r\sinh s+\nu s)}ds\\&=:\tilde{J}_\nu(r)-E_\nu(r).
\end{split}
\end{equation}
Clearly, $E_\nu(r)=0$ when $\nu\in\Z^+$.  An easy computation shows that
\begin{equation}\label{bessele}
|E_\nu(r)|=\Big|\frac{\sin(\nu\pi)}{\pi}\int_0^\infty e^{-(r\sinh
s+\nu s)}ds\Big|\leq C (r+\nu)^{-1}.
\end{equation}
There is a number of references for the asymptotic behavior of a Bessel function, see e.g. \cite{Stein1,Stempak, Watson, CL}.
We recall some properties of a Bessel
function for a convenience.
\begin{lemma}[Asymptotics of Bessel functions] \label{Bessel} Let $\nu\gg1$ and let $J_\nu(r)$ be
the Bessel function of order $\nu$ defined as above. Then there exists
a large constant $C$ and small constant $c$ independent of $\nu$ and
$r$ such that:

\begin{itemize}
\item  When $r\leq \frac \nu2$, we have
\begin{equation}\label{b1}
|J_\nu(r)|\leq C e^{-c(\nu+r)};
\end{equation}

\item  When $\frac \nu 2\leq r\leq 2\nu$, we have
\begin{equation}\label{b2}
 |J_\nu(r)|\leq C
\nu^{-\frac13}(\nu^{-\frac13}|r-\nu|+1)^{-\frac14};
\end{equation}

\item  When $r\geq 2\nu$, we have
\begin{equation}\label{b3}
 J_\nu(r)=r^{-\frac12}\sum_{\pm}a_\pm(\nu,r) e^{\pm ir}+E(\nu,r),
\end{equation}
\end{itemize}
where $|a_\pm(\nu,r)|\leq C$ and $|E(\nu,r)|\leq Cr^{-1}$.
\end{lemma}

\section{Proof of Theorem \ref{thm}}
In this section, we prove Theorem \ref{thm} by using some localized linear estimates whose proof are postpone to the next section. Since inequality \eqref{q0} is a special case of \eqref{main}, we aim to prove \eqref{main}.
Since \eqref{main} is a direct consequence of
the Stein-Tomas inequality \cite{Stein1} for the case  $p\leq2$, it suffices to prove \eqref{main} for  the case  $p\geq2$.
 More precisely, we
will only establish the estimate for $q>{2(n+1)}/{n}$, $(n+2)/q={n}/{p'}$ with $p\geq
2$
\begin{equation}\label{aim}
\|(fd\sigma)^{\vee}\|_{L^q_{t,x}(\R\times\R^n)}\leq
C_{p,q,n,S}\|(1+|\Omega|\big)^{s}f\|_{L^p(S;d\sigma)}.
\end{equation}
Recall the notation $\mathbb{L}_M$ and $A_R$ in the subsection 2.1. We decompose $f$ into a sum of dyadic supported
functions
\begin{equation*}
\begin{split}
f=\sum_{M}f_M,
\end{split}
\end{equation*}
where $f_M=f\chi_{\{(\tau,\xi):\tau=|\xi|^2, M\leq|\xi|\leq2M\}}\in
{\mathbb L}_{M}$. It follows that
\begin{equation}\label{decom}
\begin{split}
\|(fd\sigma)^{\vee}\|_{L^q_{t,x}(\R\times\R^n)}
=&\bigg\|\sum_{M}(f_Md\sigma)^{\vee}\bigg\|_{L^q_{t,x}(\R\times\R^n)}\\
=&\bigg(\sum_{R}\Big\|\sum_{M}(f_Md\sigma)^{\vee}\Big\|^q_{L^q_{t,x}(\R\times A_R)}\bigg)^{\frac1q}\\
\lesssim&
\bigg(\sum_{R}\Big(\sum_{M}\left\|(f_Md\sigma)^{\vee}\right\|_{L^q_{t,x}(\R\times A_R)}\Big)^q\bigg)^{\frac1q}.
\end{split}
\end{equation}
To prove \eqref{aim}, we need localized linear restriction estimates.
\begin{proposition}\label{linear estimates}
Assume $f\in {\mathbb L}_1$ and $R>0$ is a dyadic number.  Then the
following linear restriction estimates hold true.\vspace{0.1cm}

$\bullet$ Let  $q=2$, then
\begin{equation}\label{lr1}
\|(fd\sigma)^\vee\|_{L^2_{t,x}(\R\times
A_R)}\lesssim\min\left\{R^\frac12, R^{\frac
n2}\right\}\|f\|_{L^2(S;d\sigma)}.
\end{equation}

$\bullet$ Let  $q=3p'$ with $2\leq p\leq4$ and $\sigma=(n-2)(\frac12-\frac1q)+\frac2q$, $0<\epsilon\ll1$, then
\begin{equation}\label{lr2}
\|(fd\sigma)^\vee\|_{L^q_{t,x}(\R\times
A_R)}\lesssim\min\left\{R^{(n-1)(\frac1q-\frac12)+\epsilon}, R^{\frac
nq}\right\}\left\|\big(1+|\Omega|\big)^{\sigma}f\right\|_{L^p(S;d\sigma)}.
\end{equation}
\end{proposition}

We postpone the proof of Proposition \ref{linear estimates} to the next section, and we complete the proof of Theorem \ref{thm} by this proposition.
By a scaling argument, we conclude from \eqref{lr1} that
\begin{equation*}
\|(f_Md\sigma)^\vee\|_{L^2_{t,x}(\R\times
A_R)}\lesssim\min\left\{(RM)^\frac12, (RM)^{\frac
n2}\right\}M^{n-\frac{n+2}2-\frac n2}\|f_M\|_{L^2(S;d\sigma)}.
\end{equation*}
For any $(q,p)$ satisfying
 $$q>{2(n+1)}/{n},\;\; (n+2)/q={n}/{p'} \;\;\;\text{with}\;\; \;p\geq
2, $$
 let $\alpha=2-\frac3q-\frac1p$, then we choose ${\bar{q}}=3{\bar{p}}'$ such that
$$\frac1q=\frac{1-\alpha}2+\frac{\alpha}{\bar{q}},\qquad \frac1p=\frac{1-\alpha}2+\frac{\alpha}{\bar{p}}. $$

From \eqref{lr2}, we have that for $\bar{q}=3\bar{p}'$ with $2\leq \bar{p}\leq 4$ and $\bar{\sigma}=(n-2)(\frac12-\frac1{\bar{q}})+\frac2{\bar{q}}$
\begin{equation*}
\begin{split}
&\|(f_Md\sigma)^\vee\|_{L^{\bar q}_{t,x}(\R\times
A_R)}\\
\lesssim&\min\left\{(RM)^{(n-1)(\frac1{\bar{q}}-\frac12)+\bar\epsilon},
(RM)^{\frac n{\bar{q}}}\right\}M^{n-\frac{n+2}{\bar{q}}-\frac
n{\bar{p}}}\left\|(1+|\Omega|\big)^{\bar{\sigma}}f_M\right\|_{L^{\bar{p}}(S;d\sigma)},
\end{split}
\end{equation*}
where $0<\bar{\epsilon}\ll1$.
Therefore we obtain by an interpolation theorem
\begin{equation}\label{Lq}
\begin{split}
&\|(f_Md\sigma)^\vee\|_{L^q_{t,x}(\R\times
A_R)}\\
\lesssim&\min\{(RM)^{\frac n q},
(RM)^{-\frac{n-1}2[1-\frac{2(n+1)}{qn}]+\epsilon}\}\left\|(1+|\Omega|\big)^{\sigma}f_M\right\|_{L^{p}(S;d\sigma)}.
\end{split}
\end{equation}
Here $0<\epsilon:=\bar{\epsilon}\alpha\ll1$.
According to \eqref{decom}, we obtain
\begin{equation*}
\begin{split}
&\|(fd\sigma)^{\vee}\|_{L^q_{t,x}(\R\times\R^n)}
\\
\lesssim&
\left(\sum_{R}\left(\sum_{M}\min\left\{(RM)^{\frac n q},
(RM)^{-\frac{n-1}2[1-\frac{2(n+1)}{qn}]+\epsilon}\right\}\|(1+|\Omega|\big)^{\sigma}f_M\|_{L^{p}(S;d\sigma)}
\right)^q\right)^{\frac1q}.
\end{split}
\end{equation*}
Since $q>{2(n+1)}/n$, $\epsilon\ll1$, and $R, M$ are both dyadic number, we have
\begin{equation*}
\begin{split}
&\sup_{R>0}\bigg(\sum_M \min\Big\{(RM)^{\frac n q},
(RM)^{-\frac{n-1}2[1-\frac{2(n+1)}{qn}]+\epsilon}\Big\}\bigg)<\infty,\\
&\sup_{M>0}\bigg(\sum_R \min\Big\{(RM)^{\frac n q},
(RM)^{-\frac{n-1}2[1-\frac{2(n+1)}{qn}]+\epsilon}\Big\}\bigg)<\infty.
\end{split}
\end{equation*}
Note that for $q>{2(n+1)}/n>p\geq2$, we have by the Schur lemma  and embedding inequality
\begin{equation*}
\begin{split}
\|(fd\sigma)^{\vee}\|_{L^q_{t,x}(\R\times\R^n)}
&\lesssim
\bigg(\sum_{M}\|(1+|\Omega|\big)^{\sigma}f_M\|^p_{L^{p}(S;d\sigma)}\bigg)^{\frac1p}\\&
=\left\|(1+|\Omega|\big)^{\sigma}f\right\|_{L^{p}(S;d\sigma)}.
\end{split}
\end{equation*}
Choosing  $q=q_0=\left(2(n+1)/n \right)_+$ and
$(n+2)/q_0={n}/{p_0'}$, we  have
\begin{equation*}
\begin{split}
\|(fd\sigma)^{\vee}\|_{L^{q_0}_{t,x}(\R\times\R^n)}
\lesssim \left\|(1+|\Omega|\big)^{\sigma_0}f\right\|_{L^{p_0}(S;d\sigma)}.
\end{split}
\end{equation*}
This implies \eqref{q0}.
Interpolating this inequality with the restriction estimate  by Bourgain-Guth\cite[Theorem 1]{BG}, we  prove \eqref{aim}. Hence, the proof of estimate \eqref{main} is completed.
\vspace{0.2cm}

\section{Localized restriction estimate}
In this section we prove  Proposition \ref{linear estimates}. We start our proof by recalling
\begin{equation}\label{Fouri}
\begin{split}
&(f(\tau,\xi)d\sigma)^\vee(t,x) =\int_{\R^{n}}
g(\xi)e^{2\pi i(x\cdot \xi+t|\xi|^2)}d\xi,
\end{split}
\end{equation}
where $g(\xi)=f(|\xi|^2, \xi)\in \mathcal{S}(\R^n)$
with  $\text{supp}~g\subset\{\xi:|\xi|\in [1,2]\}$. We apply the
spherical harmonic expansion to $g$ to obtain $$
g(
\xi)=\sum_{k=0}^{\infty}\sum_{\ell=1}^{d(k)}a_{k,\ell}(\rho)Y_{k,\ell}(\omega).$$
 Recalling  $\nu(k)=k+(n-2)/2$, we have by \eqref{Fourier}
\begin{equation}\label{Fouri1}
\begin{split}
&(fd\sigma)^\vee(t,x)
=2\pi r^{-\frac{n-2}2}\sum_{k=0}^{\infty}\sum_{\ell=1}^{d(k)}
i^{k}Y_{k,\ell}(\theta) \int_0^\infty e^{-2\pi it\rho^2}
J_{\nu(k)}(2\pi r\rho)a_{k,\ell}(\rho)\rho^{\frac
n2}\varphi(\rho)d\rho.
\end{split}
\end{equation}
Here we insert a harmless smooth bump function $\varphi$ supported on the interval
$(1/2, 4)$ into the above integral, since $a_{k,\ell}(\rho)$ is
supported on $[1,2]$. Now we estimate the quantity $\|(fd\sigma)^\vee\|_{L^q_{t,x}(\R\times
A_R)}$. To this end, we first prove the following lemma.

\begin{lemma}\label{lem} Let $\mu(r)=r^{n-1}dr$ and $\omega(k)$ be a weight specified below. For $q\geq2$, we have
\begin{equation}\label{a3.7}
\begin{split}
&\bigg\|r^{-\frac{n-2}2}\Big(\sum_{k=0}^\infty
\sum_{\ell=1}^{d(k)}\omega(k)\big|\int_0^\infty e^{it\rho^2}
J_{\nu(k)}( r\rho)a_{k,\ell}(\rho)\varphi(\rho)\rho^{
\frac{n-2}2}\rho d\rho\big|^2\Big)^{\frac12}
\bigg\|_{L^q_t(\R;L^q_{\mu(r)}(S_R))}\\
\lesssim & \bigg\|r^{-\frac{n-2}2}\Big(\sum_{k=0}^\infty
\sum_{\ell=1}^{d(k)}\omega(k)\big\|J_{\nu(k)}(
r\rho)a_{k,\ell}(\rho)\varphi(\rho)\rho^{
\frac{n-2}2+\frac1{q'}}\big\|^2_{L^{q'}_\rho}\Big)^{\frac12}
\bigg\|_{L^q_{\mu(r)}(S_R)}.
\end{split}
\end{equation}
\end{lemma}

\begin{proof}Since $q\geq 2$, the Minkowski
inequality and the Fubini theorem show that the left hand side of
\eqref{a3.7} is bounded by
\begin{equation*}
\begin{split}
\bigg\|r^{-\frac{n-2}2}\Big(\sum_{k=0}^\infty\sum_{\ell=1}^{d(k)}\omega(k)\Big\|\int_0^\infty
e^{it\rho^2} J_{\nu(k)}( r\rho)a_{k,\ell}(\rho)\varphi(\rho)\rho^{\frac
{n-2}2}\rho~d\rho\Big\|^2_{L^q_t(\R)}\Big)^{\frac12}\bigg\|_{L^q_{\mu(r)}(S_R)}.
\end{split}
\end{equation*}
We rewrite this by making the variable change $\rho^2\rightsquigarrow\rho$
\begin{equation}\label{var}
\begin{split}
\bigg\|r^{-\frac{n-2}2}\Big(\sum_{k=0}^\infty\sum_{\ell=1}^{d(k)}\omega(k)\Big\|\int_0^\infty
e^{it\rho} J_{\nu(k)}(
r\sqrt{\rho})a_{k,\ell}(\sqrt{\rho})\varphi(\sqrt{\rho})\rho^{\frac
{n-2}4}~d\rho\Big\|^2_{L^q_t(\R)}\Big)^{\frac12}\bigg\|_{L^q_{\mu(r)}(S_R)}.
\end{split}
\end{equation}
We use the Hausdorff-Young inequality with respect to $t$ and we change
variables back to obtain
\begin{equation*}
\begin{split}
\text{LHS of }~\eqref{a3.7}\lesssim
\Big\|r^{-\frac{n-2}2}\Big(\sum_{k=0}^\infty\sum_{\ell=1}^{d(k)}\omega(k)
\big\|J_{\nu(k)}(
r\rho)a_{k,\ell}(\rho)\varphi(\rho)\rho^{(n-2)/2+1/q'}\big\|_{L^{q'}_\rho}^2\Big)^{\frac12}
\Big\|_{L^q_{\mu(r)}(S_R)}.
\end{split}
\end{equation*}
\end{proof}
Now we prove that the inequalities \eqref{lr1} and \eqref{lr2} with $ R\lesssim1$. For doing
this, we need
\begin{lemma}\label{lem1} Let $q\geq2$ and $R\lesssim1$, we have the following estimate
\begin{equation}\label{R<1}
\begin{split}
\|(f~d\sigma)^\vee\|_{L^q_{t,x}(\R\times
A_R)}\lesssim R^{\frac nq}
\bigg(\sum_{k=0}^\infty\sum_{\ell=1}^{d(k)}\omega(k)\big\|a_{k,\ell}(\rho)\varphi(\rho)\big\|_{L^{q'}_{\rho}}^2\bigg)^{\frac12},\end{split}
\end{equation}
where $\omega(k)=(1+k)^{2(n-1)(1/2-1/q)}$.

\end{lemma}
We postpone the proof of this lemma for a moment. Note that for $q'\leq 2\leq p$,  we use \eqref{R<1}, \eqref{a2.3}, the Minkowski inequality and the H\"older inequality
to obtain
\begin{equation*}
\begin{split}
\|(f~d\sigma)^\vee\|_{L^q_{t,x}(\R\times
A_R)}
\lesssim& R^{\frac nq}
\bigg\|\Big(\sum_{k=0}^\infty\sum_{\ell=1}^{d(k)}\omega(k)\big|a_{k,\ell}(\rho)\big|^2\Big)^{\frac12}\varphi(\rho)\bigg\|_{L^{q'}_{\rho}}\\
\lesssim& R^{\frac nq}
\left\|g\right\|_{L^{q'}_{\rho}H_\omega^{m}(\mathbb{S}^{n-1})}\lesssim
R^{\frac nq}
\left\|g\right\|_{L^{p}_{\rho}H_\omega^{m,p}(\mathbb{S}^{n-1})},
\end{split}
\end{equation*}
where $m=(n-1)(\frac12-\frac1q)$. In particular, for $q=2$ and $4\leq q\leq6$, this proves \eqref{lr1} and \eqref{lr2} when $R\lesssim 1$.
Hence it suffices to consider the case $R\gg1$ once we prove Lemma
\ref{lem1}.
\begin{proof}[Proof of Lemma~\ref{lem1}]  By scaling argument in variables $t,x$ and \eqref{Fouri1}, we obtain
\begin{equation}
\begin{split}
&\|(f~d\sigma)^\vee\|_{L^q_{t,x}(\R\times
A_R)}\\
\lesssim& \bigg\| r^{-\frac{n-2}2}\sum_{k=0}^{\infty}\sum_{\ell=1}^{d(k)}
i^{k}Y_{k,\ell}(\theta) \int_0^\infty e^{- it\rho^2}
J_{\nu(k)}( r\rho)a_{k,\ell}(\rho)\rho^{\frac
n2}\varphi(\rho)~d\rho
\bigg\|_{L^q_{t,x}(\R\times
A_R)}.
\end{split}
\end{equation}
By Sobolev's embedding, \eqref{2.3} and \eqref{a2.3}, we have
\begin{equation*}
\begin{split}
&\|(f~d\sigma)^\vee\|_{L^q_{t,x}(\R\times A_R)}\\
\lesssim&
\bigg\|r^{-\frac{n-2}2}\Big(\sum_{k=0}^\infty
\sum_{\ell=1}^{d(k)}\omega(k)\Big|\int_0^\infty e^{it\rho^2}
J_{\nu(k)}( r\rho)a_{k,\ell}(\rho)\varphi(\rho)\rho^{
\frac{n-2}2}\rho~d\rho\Big|^2\Big)^{\frac12}
\bigg\|_{L^q_t(\R;L^q_{\mu(r)}(S_R))}.
\end{split}
\end{equation*}
By Lemma \ref{lem}, it is enough to show
\begin{equation*}
\begin{split}
&\Big\|r^{-\frac{n-2}2}\Big(\sum_{k=0}^\infty\sum_{\ell=1}^{d(k)}\omega(k)
\big\|J_{\nu(k)}(
r\rho)a_{k,\ell}(\rho)\varphi(\rho)\rho^{(n-2)/2+1/q'}\big\|_{L^{q'}_\rho}^2\Big)^{\frac12}
\Big\|_{L^q_{\mu(r)}(S_R)}\\
\lesssim& R^{\frac nq}
\bigg(\sum_{k=0}^\infty\sum_{\ell=1}^{d(k)}\omega(k)\big\|a_{k,\ell}(\rho)\varphi(\rho)\big\|_{L^{q'}_{\rho}}^2\bigg)^{\frac12}.
\end{split}
\end{equation*}
Writing briefly $\nu=\nu(k)$,  and noting that $R<r<2R$ and $1<\rho<2$, we have by \eqref{rf}
\begin{equation*}
\begin{split}
&\Big\|r^{-\frac{n-2}2}\Big(\sum_{k=0}^\infty\sum_{\ell=1}^{d(k)}\omega(k)\big\|J_{\nu(k)}(
r\rho)a_{k,\ell}(\rho)\varphi(\rho)\rho^{(n-2)/2+1/q'}\big\|_{L^{q'}_\rho}^2\Big)^{\frac12}
\Big\|_{L^q_{\mu(r)}([R,2R])} \\
\lesssim &\bigg(\int_{R}^{2R}
r^{-\frac{(n-2)q}2}\Big(\sum_{k=0}^\infty\sum_{\ell=1}^{d(k)}
\omega(k)\Big|\frac{(4
r)^{\nu}}{2^{\nu}\Gamma(\nu+\frac12)\Gamma(\frac12)}\Big|^2\big\|a_{k,\ell}(\rho)\rho^\nu\varphi(\rho)\big\|_{L^{q'}_\rho}^2\Big)^{\frac
q2} r^{n-1}~dr\bigg)^{\frac1q}\\
\lesssim & R^{\frac
nq}\bigg(\sum_{k=0}^\infty\sum_{\ell=1}^{d(k)} \omega(k)\Big[\frac{(2
R)^{\nu-\frac{n-2}2}}{\Gamma(\nu+\frac12)}\Big]^2\big\|a_{k,\ell}(\rho)\rho^\nu\varphi(\rho)\big\|_{L^{q'}_\rho}^2\bigg)^{\frac
12}\\
\lesssim &R^{\frac
nq}\bigg(\sum_{k=0}^\infty\sum_{\ell=1}^{d(k)}
\omega(k)\big\|a_{k,\ell}(\rho)\varphi(\rho)\big\|_{L^{q'}_\rho}^2\bigg)^{\frac
12} .
\end{split}
\end{equation*}
In the last inequality, we use the Stirling formula
$\Gamma\left(\nu+1\right)\sim \sqrt{\nu}(\nu/e)^\nu$ and the fact
that $R\lesssim 1$ and $\nu\geq (n-2)/2$.
\end{proof}

Now we are in a position to prove Proposition \ref{linear estimates} when $R\gg1$.
We first prove \eqref{lr1} by making use of  \eqref{Fouri}.
Since
$\text{supp}~g\subset\{\xi:|\xi|\in [1,2]\}$, we may assume
$|\xi_n|\sim1$. Then we freeze one spatial variable, say $x_n$, with
$|x_n|\lesssim R$ and free other spatial variables $x'=(x_1,\ldots,
x_{n-1})$. After making the change of variables $\eta_j=\xi_j,~
\eta_n=|\xi|^2$ with $j=1,\ldots n-1$, we use the Plancherel theorem
on the spacetime Fourier transform in $(t,x')$ to obtain
\eqref{lr1}.\vspace{0.2cm}

When $R\gg1$,  inequality \eqref{lr2} is a consequence of the  interpolation theorem  and the following proposition.
\begin{proposition}\label{Local}
Assume $f\in \mathbb{L}_1$ and $R\gg1$ is a dyadic number.  For every
small constant $0<\epsilon\ll1$, we have the following inequalities \vspace{0.1cm}

$\bullet$ For $q=4$, we have
\begin{equation}\label{q=4}
\|(f~d\sigma)^\vee\|_{L^4_{t,x}(\R\times
A_R)}\lesssim R^{-\frac{n-1}4+\epsilon}
\|(1+|\Omega|\big)^{\frac{n}4}f\|_{L^4(S;~d\sigma)}.
\end{equation}

$\bullet$ For $q=6$, we have
\begin{equation}\label{q=6}
\|(f~d\sigma)^\vee\|_{L^6_{t,x}(\R\times A_R)}\lesssim R^{-\frac{n-1}3+\epsilon} \|\big(1+|\Omega|\big)^{\frac{n-1}{3}}
f\|_{L^2(S;~d\sigma)}.
\end{equation}
\end{proposition}

\begin{remark}
It seems to be possible to remove the $\epsilon$-loss in \eqref{q=6}, but we do not purchase this option here because we do not need it in this paper.

\end{remark}

To prove this proposition, we firstly show
\begin{lemma} Assume $f\in {\mathbb L}_1$ and $R\gg1$. We have the following estimate
\begin{equation}\label{R>1}
\begin{split}
\|(f~d\sigma)^\vee\|_{L^4_{t,x}(\R\times
A_R)}\lesssim R^{-\frac{n-1}4+\epsilon}
\|g\|_{L_\rho^4 H_\omega^{\frac n4,4}(\mathbb{S}^{n-1})},\end{split}
\end{equation}
where $0<\epsilon\ll1$,  and $g(\xi)=f(|\xi|^2,\xi)$.
\end{lemma}

\begin{proof} By the scaling argument and \eqref{Fouri1}, it suffices to estimate the quantity
\begin{equation}
\begin{split}
 \bigg\| r^{-\frac{n-2}2}\sum_{k=0}^{\infty}\sum_{\ell=1}^{d(k)}
i^{k}Y_{k,\ell}(\theta) \int_0^\infty e^{- it\rho^2}
J_{\nu(k)}( r\rho)a_{k,\ell}(\rho)\rho^{\frac
n2}\varphi(\rho)~d\rho
\bigg\|_{L^4_{t,x}(\R\times
A_R)}.
\end{split}
\end{equation}
In the following, we consider the three cases. For the first two cases, we establish the estimates for general $q\geq 4$ so that
we can use them directly for $q=6$ later.

$\bullet$ Case 1: $k\in \Omega_1:=\{k:R\ll \nu(k)\}$.
Let $\omega(k)=(1+k)^{2(n-1)(1/2-1/q)}$ again. We have by a similar argument as in the proof of Lemma \ref{lem1}:
\begin{equation*}
\begin{split}
& \bigg\| r^{-\frac{n-2}2}\sum_{k\in\Omega_1}\sum_{\ell=1}^{d(k)}
i^{k}Y_{k,\ell}(\theta) \int_0^\infty e^{- it\rho^2} J_{\nu(k)}(
r\rho)a_{k,\ell}(\rho)\rho^{\frac n2}\varphi(\rho)~d\rho
\bigg\|_{L^q_{t,x}(\R\times A_R)}\\
\lesssim&
\bigg\|r^{-\frac{n-2}2}\Big(\sum_{k\in\Omega_1}
\sum_{\ell=1}^{d(k)}\omega(k)\Big|\int_0^\infty e^{it\rho^2}
J_{\nu(k)}( r\rho)a_{k,\ell}(\rho)\varphi(\rho)\rho^{
\frac{n-2}2}\rho~d\rho\Big|^2\Big)^{\frac12}
\bigg\|_{L^q_t(\R;L^q_{\mu(r)}(S_R))}\\
\lesssim&
\bigg\|r^{-\frac{n-2}2}\Big(\sum_{k\in\Omega_1}\sum_{\ell=1}^{d(k)}\omega(k)
\big\|J_{\nu(k)}(
r\rho)a_{k,\ell}(\rho)\varphi(\rho)\rho^{(n-2)/2+1/q'}\big\|_{L^{q'}_\rho}^2\Big)^{\frac12}
\bigg\|_{L^q_{\mu(r)}(S_R)}.
\end{split}
\end{equation*}
Recall that for $R\gg1$ and $k\in\Omega_1$,  we have
$|J_{\nu(k)}(r)|\lesssim e^{-c(r+\nu)}$  by  \eqref{b1}. Using $R<r<2R$ and
$1<\rho<2$, we obtain
\begin{equation*}
\begin{split}
&\Big\|r^{-\frac{n-2}2}\Big(\sum_{k\in\Omega_1}\sum_{\ell=1}^{d(k)}\omega(k)\big\|J_{\nu(k)}(
r\rho)a_{k,\ell}(\rho)\varphi(\rho)\rho^{(n-2)/2+1/q'}\big\|_{L^{q'}_\rho}^2\Big)^{\frac12}
\Big\|_{L^q_{\mu(r)}([R,2R])} \\
\lesssim &\bigg(\int_{R}^{2R}
r^{-\frac{(n-2)q}2}\Big(\sum_{k\in\Omega_1}\sum_{\ell=1}^{d(k)}
\omega(k)e^{-(r+\nu)}\big\|a_{k,\ell}(\rho)\rho^\nu\varphi(\rho)\big\|_{L^{q'}_\rho}^2\Big)^{\frac
q2} r^{n-1}~dr\bigg)^{\frac1q}\\
 \lesssim &
e^{-cR}\bigg(\sum_{k\in\Omega_1}\sum_{\ell=1}^{d(k)}
\omega(k)e^{-\nu(k)}
\big\|a_{k,\ell}(\rho)\rho^\nu\varphi(\rho)\big\|_{L^{q'}_\rho}^2\bigg)^{\frac
12}\\\lesssim& e^{-cR}\bigg(\sum_{k\in\Omega_1}\sum_{\ell=1}^{d(k)}
\omega(k)\big\|a_{k,\ell}(\rho)\varphi(\rho)\big\|_{L^{q'}_\rho}^2\bigg)^{\frac
12} .
\end{split}
\end{equation*}
By Minkowski's inequality and H\"older's inequality, we obtain
\begin{equation}\label{R<k}
\begin{split}
& \bigg\| r^{-\frac{n-2}2}\sum_{k\in\Omega_1}\sum_{\ell=1}^{d(k)}
i^{k}Y_{k,\ell}(\theta) \int_0^\infty e^{- it\rho^2} J_{\nu(k)}(
r\rho)a_{k,\ell}(\rho)\rho^{\frac n2}\varphi(\rho)~d\rho
\bigg\|_{L^q_{t,x}(\R\times A_R)} \\
\lesssim& e^{-cR}
\bigg\|\Big(\sum_{k=0}^\infty\sum_{\ell=1}^{d(k)}\omega(k)\big|a_{k,\ell}(\rho)\big|^2\Big)^{\frac12}\varphi(\rho)\bigg\|_{L^{p}_{\rho}}
.
\end{split}
\end{equation}
Applying  this with $q=4=p$, we have
\begin{equation*}
\begin{split}
& \bigg\| r^{-\frac{n-2}2}\sum_{k\in\Omega_1}\sum_{\ell=1}^{d(k)}
i^{k}Y_{k,\ell}(\theta) \int_0^\infty e^{- it\rho^2} J_{\nu(k)}(
r\rho)a_{k,\ell}(\rho)\rho^{\frac n2}\varphi(\rho)~d\rho
\bigg\|_{L^4_{t,x}(\R\times A_R)} \\
 \lesssim & e^{-cR}
\bigg\|\Big(\sum_{k=0}^\infty\sum_{\ell=1}^{d(k)}(1+k)^{(n-1)/2}\big|a_{k,\ell}(\rho)\big|^2\Big)^{\frac12}
\varphi(\rho)\bigg\|_{L^{4}_{\rho}} \\
\lesssim&
R^{-\frac{n-1}4+\epsilon}\|g\|_{L_\rho^4
H_\omega^{(n-1)/4,4}(\mathbb{S}^{n-1})}.
\end{split}
\end{equation*}

$\bullet$ Case 2: $k\in \Omega_2:=\{k: \nu(k)\sim R
\}$. Recalling $g(\xi)=f(|\xi|^2, \xi)$,  and using the Sobolev embedding, the Strichartz estimate and the fact $\text{supp}~g\subset\{\xi\in\R^n:|\xi|\in [1,2]\}$,
we have for $q\geq4$ and $\frac2q=n(\frac12-\frac1r)$
\begin{equation}\label{str}
\begin{split}
\|(f~d\sigma)^{\vee}\|_{L^q_{t,x}(\R\times\R^n)}\lesssim \|(f~d\sigma)^{\vee}\|_{L^q(\R; H^m_{r}(\R^n))}\lesssim\| \hat g\|_{H^{m}(\R^n)}\lesssim \| g\|_{L^2(\R^n)}
\end{split}
\end{equation}
where $m=\frac{(q-2)n-4}{2q}\geq0$ since $n\geq2$.
If $g=\bigoplus_{k\in\Omega_2} \mathcal{H}^{k}$, then
\begin{equation}
\begin{split}
\|g\|^2_{L_\omega^2(\mathbb{S}^{n-1})}=&\sum_{k\in\Omega_2}\sum_{\ell=1}^{{d}(k)}|a_{k,\ell}|^2\\\lesssim&
R^{-2(n-1)(1/2-1/q)}
\sum_{k\in\Omega_2}\sum_{\ell=1}^{{d}(k)}(1+k)^{2(n-1)(1/2-1/q)}|a_{k,\ell}|^2
\\\lesssim& R^{-2(n-1)(1/2-1/q)}\|g\|^2_{H_\omega^{(n-1)(\frac12-\frac1q),2}(\mathbb{S}^{n-1})} .\end{split}
\end{equation}
Since $\text{supp} g\subset \{\xi\in\R^n: |\xi|\in [1,2]\}$ and $p\geq2$,  we have by H\"older's inequality and \eqref{str}
\begin{equation}\label{Rk}
\begin{split}
&\Big\| r^{-\frac{n-2}2}\sum_{k\in\Omega_2}\sum_{\ell=1}^{d(k)}
i^{k}Y_{k,\ell}(\theta) \int_0^\infty e^{- it\rho^2}
J_{\nu(k)}( r\rho)a_{k,\ell}(\rho)\rho^{\frac
n2}\varphi(\rho)~d\rho
\Big\|_{L^q_{t,x}(\R\times
A_R)}\\
 \lesssim&
R^{-(n-1)(1/2-1/q)}\|g\|_{L_\rho^p H_\omega^{(n-1)(\frac12-\frac1q),p}(\mathbb{S}^{n-1})}.
\end{split}
\end{equation}
In particular, when $q=p=4$, inequality \eqref{Rk} implies that
\begin{equation}
\begin{split}
&\Big\| r^{-\frac{n-2}2}\sum_{k\in\Omega_2}\sum_{\ell=1}^{d(k)}
i^{k}Y_{k,\ell}(\theta) \int_0^\infty e^{- it\rho^2}
J_{\nu(k)}( r\rho)a_{k,\ell}(\rho)\rho^{\frac
n2}\varphi(\rho)~d\rho
\Big\|_{L^4_{t,x}(\R\times
A_R)}\\
 \lesssim&
R^{-(n-1)/4}\|g\|_{L_\rho^4 H_\omega^{(n-1)/4,4}(\mathbb{S}^{n-1})}.
\end{split}
\end{equation}

$\bullet$ Case 3: $k\in \Omega_3:=\{k:  \nu(k)\ll R\}$. We need the following lemma about the oscillation and decay property of a Bessel function.
This lemma was proved by
Barcelo-Cordoba \cite{BC}.

\begin{lemma}[Oscillation and asymptotic property, \cite{BC}]  Let $\nu>1/2$ and $r>\nu+\nu^{1/3}$. There exists a constant number $C$ independent of $r$
and $\nu$ such that
\begin{equation}
J_{\nu}(r)=\sqrt{\frac2{\pi}}\frac{\cos\theta(r)}{(r^2-\nu^2)^{1/4}}+h_\nu(r),
\end{equation}
where $\theta(r)=(r^2-\nu^2)^{1/2}-\nu\arccos\frac{\nu}r-\frac\pi4$ and
\begin{equation}
|h_{\nu}(r)|\leq
C\left(\left(\frac{\nu^2}{(r^2-\nu^2)^{7/4}}+\frac1r\right)1_{[\nu+\nu^{1/3},2\nu]}(r)+\frac1r1_{[2\nu,\infty)}(r)\right).
\end{equation}
\end{lemma}

\vskip0.12cm

Note that $\nu(k)=k+(n-2)/2$ and $k\in\Omega_3$,   we can write
$$J_\nu(r)=I_{\nu}(r)+\bar{I}_{\nu}(r)+h_\nu(r), \quad \;\text{where}\;\; |h_\nu(r)|\lesssim r^{-1},$$
 and
$$I_{\nu}(r)=\frac{\sqrt{2/\pi}e^{i\theta(r)}}{\left(r^2-\nu^2\right)^{1/4}}.$$
A simple computation yields to
\begin{equation}\label{dtheta}
\left\{\begin{aligned}
&\theta'(r)=(r^2-\nu^2)^{1/2}r^{-1},\\
&\theta''(r)=(r^2-\nu^2)^{-1/2}-(r^2-\nu^2)^{1/2}r^{-2}=(r^2-\nu^2)^{-1/2}\nu^2r^{-2},\\
&\theta'''(r)=\frac{\nu^2}r(r^2-\nu^2)^{-3/2}\nu^2r^{-2}\left(-3+\frac{2\nu^2}{r^2}\right).
\end{aligned}\right.
\end{equation}
Using Sobolev embedding on sphere and Minkowski's inequality, we
estimate
\begin{equation*}
\begin{split}
& \bigg\| r^{-\frac{n-2}2}\sum_{k\in\Omega_3}\sum_{\ell=1}^{d(k)}
i^{k}Y_{k,\ell}(\theta) \int_0^\infty e^{- it\rho^2} J_{\nu(k)}(
r\rho)a_{k,\ell}(\rho)\rho^{\frac n2}\varphi(\rho)~d\rho
\bigg\|_{L^4_{t,x}(\R\times A_R)}\\
\lesssim&
R^{-\frac{n-2}2}\bigg\|
\Big(\sum_{k\in\Omega_3}\sum_{\ell=1}^{d(k)}(1+k)^{(n-1)/2}
\Big|\int_0^\infty e^{- it\rho^2} J_{\nu(k)}(
r\rho)a_{k,\ell}(\rho)\rho^{\frac
n2}\varphi(\rho)~d\rho\Big|^2\Big)^{1/2}
\bigg\|_{L^4_t(\R;L^4_{\mu(r)}(S_R))}
\\
\lesssim& R^{-\frac{n-3}4}\bigg\|\bigg(\sum_{k\in\Omega_3}\sum_{\ell=1}^{d(k)}(1+k)^{(n-1)/2}
\Big|\int_0^\infty e^{- it\rho^2} J_{\nu(k)}(
r\rho)a_{k,\ell}(\rho)\rho^{\frac
n2}\varphi(\rho)~d\rho\Big|^2\Big)^{1/2}
\bigg\|_{L^4_t(\R;L^4_{r}(S_R))}.
\end{split}
\end{equation*}
Since  $J_\nu(r)=I_{\nu}(r)+\bar{I}_{\nu}(r)+h_\nu(r)$, it suffices to estimate two terms
\begin{equation}\label{hv}
\begin{split}
&\bigg(\sum_{k\in\Omega_3}\sum_{\ell=1}^{d(k)}(1+k)^{(n-1)/2}
 \Big\|\int_0^\infty e^{- it\rho^2}
h_{\nu(k)}( r\rho)a_{k,\ell}(\rho)\rho^{\frac
n2}\varphi(\rho)~d\rho\Big\|_{L^4_t(\R;L^4_{r}(S_R))}^{2}\bigg)^{1/2}\\&\lesssim  R^{-3/4}\|g\|_{L_\rho^4
H_\omega^{\frac{n-1}4,4}(\mathbb{S}^{n-1})}
\end{split}
\end{equation}
and \begin{equation}\label{Iv}
\begin{split}
&\bigg\|\bigg(\sum_{k\in\Omega_3}\sum_{\ell=1}^{d(k)}(1+k)^{(n-1)/2}
\Big|\int_0^\infty e^{- it\rho^2} I_{\nu(k)}(
r\rho)a_{k,\ell}(\rho)\rho^{\frac
n2}\varphi(\rho)~d\rho\Big|^2\Big)^{1/2}
\bigg\|_{L^4_t(\R;L^4_{r}(S_R))}\\&\lesssim  R^{-1/2+\epsilon}\|g\|_{L_\rho^4
H_\omega^{\frac{n}4,4}(\mathbb{S}^{n-1})}.
\end{split}
\end{equation}
For the first purpose, we consider the operator
$$T_{\nu}(a)(t,r)=\chi\Big(\frac r R\Big)\int_0^\infty
e^{- it\rho^2} h_{\nu}(
r\rho)a(\rho)\rho^{\frac n2}\varphi(\rho) d\rho$$
where $|h_\nu(r)|\leq C/r$.  By a similar argument as in the proof of Lemma \ref{lem}, it is easy to see
\begin{equation}\label{T}
\|T_{\nu}(a)(t,r)\|_{L^q_{t,r}}\leq
R^{-1/q'}\|a\varphi\|_{L^{q'}_\rho}.
\end{equation}
Hence we have
\begin{equation*}
\begin{split}
&\Big(\sum_{k\in\Omega_3}\sum_{\ell=1}^{d(k)}(1+k)^{(n-1)/2}
 \Big\|\int_0^\infty e^{- it\rho^2}
h_{\nu(k)}( r\rho)a_{k,\ell}(\rho)\rho^{\frac
n2}\varphi(\rho)~d\rho\Big\|_{L^4_t(\R;L^4_{r}(S_R))}^{2}\Big)^{1/2}\\
\lesssim&
 R^{-3/4}\Big(\sum_{k\in\Omega_3}\sum_{\ell=1}^{d(k)}(1+k)^{(n-1)/2}
\Big\|a_{k,\ell}(\rho)\varphi(\rho)\Big\|^2_{L^{4/3}}\Big)^{1/2}\\
\lesssim&
 R^{-3/4}\Big\|\Big(\sum_{k\in\Omega_3}\sum_{\ell=1}^{d(k)}(1+k)^{(n-1)/2}
\big|a_{k,\ell}(\rho)\big|^2\Big)^{1/2}\varphi \Big\|_{L^{4/3}}\\
\lesssim & R^{-3/4}\|g\|_{L_\rho^4 H_\omega^{\frac{n-1}4,4}(\mathbb{S}^{n-1})}
\end{split}
\end{equation*}
which implies \eqref{hv}.

Next we prove \eqref{Iv}. To this end,  let $\beta(\rho)=\rho^{\frac n2}\varphi(\rho)$, we see that
\begin{equation}\label{Iv'}
\begin{split}
&\left\|\left(\sum_{k\in\Omega_3}\sum_{\ell=1}^{d(k)}(1+k)^{(n-1)/2}
 \left|\int_0^\infty e^{- it\rho^2}
I_{\nu(k)}( r\rho)a_{k,\ell}(\rho)\rho^{\frac
n2}\varphi(\rho)\mathrm{d}\rho\right|^{2}\right)^{1/2}\right\|^4_{L^4_t(\R;L^4_{r}(S_R))}
\\&=\Big\|\sum_{k\in\Omega_3}\sum_{\ell=1}^{d(k)}(1+k)^{(n-1)/2}
\int_{\R^2}e^{- it(\rho_1^2-\rho_2^2)}
I_{\nu(k)}( r\rho_1)\overline{I_{\nu(k)}( r\rho_2)}\\&\qquad\qquad\times a_{k,\ell}(\rho_1)\overline{a_{k,\ell}(\rho_2)}\beta(\rho_1)\beta(\rho_2)\mathrm{d}\rho_1\mathrm{d}\rho_2\Big\|^2_{L^2_t(\R;L^2_{r}(S_R))}\\
&\leq \Big(\sum_{k\in\Omega_3}(1+k)^{(n-1)/2}
\Big\|\int_{\R^2}e^{- it(\rho_1^2-\rho_2^2)}
I_{\nu(k)}( r\rho_1)\overline{I_{\nu(k)}( r\rho_2)}\\&\qquad\qquad\times \sum_{\ell=1}^{d(k)} a_{k,\ell}(\rho_1)\overline{a_{k,\ell}(\rho_2)}\beta(\rho_1)\beta(\rho_2)\mathrm{d}\rho_1\mathrm{d}\rho_2\Big\|_{L^2_t(\R;L^2_{r}(S_R))}\Big)^2\\
&=\Big(\sum_{k\in\Omega_3}(1+k)^{(n-1)/2} \Big( \int_{\R^4} \sum_{\ell=1}^{d(k)} a_{k,\ell}(\rho_1)
\overline{a_{k,\ell}(\rho_2)}\sum_{\ell'=1}^{d(k)}\overline{a_{k,\ell'}(\rho_3)}a_{k,\ell'}(\rho_4)\beta(\rho_1)\beta(\rho_2)\beta(\rho_3)\beta(\rho_4)  \\&
\int_\R e^{- it(\rho_1^2-\rho_2^2+\rho_3^2-\rho_4^2)} dt K(R,\nu;\rho_1,\rho_2,\rho_3,\rho_4)   d\rho_1d\rho_2d\rho_3d\rho_4\Big)^{1/2}\Big)^2\\
\end{split}
\end{equation}
where the kernel
\begin{equation}\label{kernel}
\begin{split}
&K(R,\nu;\rho_1,\rho_2,\rho_3,\rho_4)\\
=&\int_0^\infty \frac{\chi(\frac r R)
e^{i\left(\theta(\rho_1r)-\theta(\rho_2r)+\theta(\rho_3r)-\theta(\rho_4r)\right)}}{\left((r\rho_1)^2-\nu^2\right)^{1/4}
\left((r\rho_2)^2-\nu^2\right)^{1/4}\left((r\rho_3)^2-\nu^2\right)^{1/4}\left((r\rho_4)^2-\nu^2\right)^{1/4}}dr.
\end{split}
\end{equation}
Now we analyze the kernel $K$. Let
$$\phi(r;\rho_1,\rho_2,\rho_3,\rho_4)=\theta(\rho_1r)-\theta(\rho_2r)+\theta(\rho_3r)-\theta(\rho_4r).$$
Hence if $\rho_1^2-\rho_2^2=\rho_4^2-\rho_3^2$, we have by
\eqref{dtheta}
\begin{equation*}
\begin{split}
\phi'_r&=(\rho_1^2-\rho_2^2)r\Big(\frac1{\sqrt{(r\rho_1)^2-\nu^2}+\sqrt{(r\rho_2)^2-\nu^2}}-\frac1{\sqrt{(r\rho_3)^2-\nu^2}
+\sqrt{(r\rho_4)^2-\nu^2}}\Big)\\
&=\frac{(\rho_1^2-\rho_2^2)(\rho_3^2-\rho_2^2)r^3}{\Big(\sqrt{(r\rho_1)^2-\nu^2}+\sqrt{(r\rho_2)^2-\nu^2}\Big)\Big(\sqrt{(r\rho_3)^2-\nu^2}
+\sqrt{(r\rho_4)^2-\nu^2}\Big)}\\&\qquad
\bigg(\frac{1}{\sqrt{(r\rho_3)^2-\nu^2}+\sqrt{(r\rho_2)^2-\nu^2}}+\frac{1}{\sqrt{(r\rho_1)^2-\nu^2}+\sqrt{(r\rho_4)^2-\nu^2}}\bigg).
\end{split}
\end{equation*}
Since $k\in\Omega_3$, one has $r\gg\nu(k)$. Therefore we have
\begin{equation*}
\begin{split}
|\phi'_r|&\geq |\rho_1^2-\rho_2^2|\cdot|\rho_3^2-\rho_2^2|.
\end{split}
\end{equation*}  Applying integration by parts with respect to  $r$ to \eqref{kernel}, we have for any $N\geq0$
\begin{equation}\label{decay}
K(R,\nu;\rho_1,\rho_2,\rho_3,\rho_4)\lesssim
R^{-1}\left(1+R|\rho_1^2-\rho_2^2|\cdot|\rho_3^2-\rho_2^2|\right)^{-N},
\end{equation}
when $\rho_1^2-\rho_2^2=\rho_4^2-\rho_3^2$.
 Let $b_{k,\ell}(\rho)=2a_{k,\ell}(\sqrt{\rho})\beta({\sqrt{\rho})}/\sqrt{\rho}$, from \eqref{Iv'} and \eqref{decay}, it suffices to estimate
\begin{equation*}
\begin{split}
&\Big(\sum_{k\in\Omega_3}(1+k)^{(n-1)/2} \Big(\int_{\R^4} \delta(\rho_1-\rho_2+\rho_3-\rho_4)
 K(R,\nu(k);\sqrt{\rho_1},\sqrt{\rho_2},\sqrt{\rho_3},\sqrt{\rho_4}) \\&\quad \times \sum_{\ell=1}^{d(k)} b_{k,\ell}(\rho_1)
\overline{b_{k,\ell}(\rho_2)}\sum_{\ell'=1}^{d(k)}\overline{b_{k,\ell'}(\rho_3)}b_{k,\ell'}(\rho_4) d\rho_1d\rho_2d\rho_3d\rho_4 \Big)^{1/2}\Big)^2\\
=&\Big(\sum_{k\in\Omega_3}(1+k)^{(n-1)/2} \Big(\int_{\R^3}
 K(R,\nu(k);\sqrt{\rho_1},\sqrt{\rho_2},\sqrt{\rho_3},\sqrt{\rho_1-\rho_2+\rho_3}) \\&\quad\times \sum_{\ell=1}^{d(k)} b_{k,\ell}(\rho_1)
\overline{b_{k,\ell}(\rho_2)}\sum_{\ell'=1}^{d(k)}\overline{b_{k,\ell'}(\rho_3)}b_{k,\ell'}(\rho_1-\rho_2+\rho_3) d\rho_1d\rho_2d\rho_3\Big)^{1/2}\Big)^2\\
\leq &R^{-1}\Big(\sum_{k\in\Omega_3}(1+k)^{(n-1)/2} \Big(\int_{\R^3}
  (1+R|\rho_1-\rho_2||\rho_3-\rho_2|)^{-N} \\&\quad\times \sum_{\ell=1}^{d(k)} \big| b_{k,\ell}(\rho_1)
\overline{b_{k,\ell}(\rho_2)}\big|\sum_{\ell'=1}^{d(k)}\big|\overline{b_{k,\ell'}(\rho_3)}b_{k,\ell'}(\rho_1-\rho_2+\rho_3)\big| d\rho_1d\rho_2d\rho_3\Big)^{1/2}\Big)^2\\
\lesssim &R^{-1}\Big(\sum_{k\in\Omega_3}(1+k)^{(n-1)/2}\Big( \int_{\R^3}
  (1+R|\rho_1-\rho_2||\rho_3-\rho_2|)^{-N}\\
\quad &\times b_k(\rho_1)b(\rho_2)b_k(\rho_3)b_k(\rho_1-\rho_2+\rho_3)d\rho_1d\rho_2d\rho_3\Big)^{1/2}\Big)^2
\end{split}
\end{equation*}
where $b_k(\rho)=\big(\sum_{\ell=1}^{d(k)} |b_{k,\ell}(\rho)|^2\big)^{1/2}$.
Then we aim to estimate
\begin{equation}\label{in-est}
\begin{split}
 \int_{\R^3}\frac{b(\rho_1)b(\rho_2)b(\rho_3)b(\rho_1-\rho_2+\rho_3)}{(1+R|\rho_1-\rho_2||\rho_3-\rho_2|)^{N}}d\rho_1d\rho_2d\rho_3\lesssim R^{-1+\epsilon}\|b\|_{L^4}^4.
\end{split}
\end{equation}
Indeed once we have proved \eqref{in-est},  we show
\begin{equation*}
\begin{split}
& \left\|\left(\sum_{k\in\Omega_3}\sum_{\ell=1}^{d(k)}(1+k)^{(n-1)/2}
 \left|\int_0^\infty e^{- it\rho^2}
I_{\nu(k)}( r\rho)a_{k,\ell}(\rho)\rho^{\frac
n2}\varphi(\rho)\mathrm{d}\rho\right|^{2}\right)^{1/2}\right\|^2_{L^4_t(\R;L^4_{r}(S_R))}\\&\lesssim  R^{-1+\epsilon}\sum_{k\in\Omega_3}(1+k)^{(n-1)/2+\frac12+\epsilon}(1+k)^{-\frac12-\epsilon}\left\|b_k\right\|^2_{L^4}
\\&\lesssim  R^{-1+2\epsilon}\Big(\sum_{k\in\Omega_3}(1+k)^{n}\Big\|(\sum_{\ell=1}^{d(k)} | b_{k,\ell}(\rho)|^2)^{1/2} \Big\|^4_{L^4}\Big)^{1/2}
\\&\lesssim  R^{-1+2\epsilon}\Big\|\big(\sum_{k\in\Omega_3}\sum_{\ell=1}^{d(k)} (1+k)^{\frac n2}| a_{k,\ell}(\rho)|^2\big)^{1/2} \Big\|^2_{L^4}
\end{split}
\end{equation*}
which implies \eqref{Iv}. Therefore, it remains to prove
\begin{equation}\label{eq-goal}
\int_{\R^3}\frac{b(\rho_1)b(\rho_2)b(\rho_3)b(\rho_1-\rho_2+\rho_3)}{\left(1+R|\rho_1-\rho_2|\cdot|\rho_3-\rho_2|\right)^N}d\rho_1
d\rho_2 d\rho_3 \lesssim R^{-1+\epsilon}\|b\|^4_{L^4}.
\end{equation}
For $R=2^{k_0}\gg1$, we decompose the integral into
\begin{equation}\label{dec}
\begin{split}
&\int_{\R^3}\frac{b(\rho_1)b(\rho_2)b(\rho_3)b(\rho_1-\rho_2+\rho_3)}{\left(1+R|\rho_1-\rho_2||\rho_3-\rho_2|\right)^N}d\rho_1 d\rho_2 d\rho_3\\
=&\bigg(\sum_{\{(i,j)\in \mathbb{N}^2; i+j\geq k_0\}}+\sum_{\{(i,j)\in \mathbb{N}^2; i+j\lesssim k_0\}}R^{-N}2^{N(i+j)}\bigg)\\
&\qquad\int b(\rho_2)
d\rho_2 \int_{|\rho_1-\rho_2|\sim 2^{-i}}b(\rho_1) d\rho_1 \int_{|\rho_3-\rho_2|\sim 2^{-j}}b(\rho_3)b(\rho_1-\rho_2+\rho_3) d\rho_3.
\end{split}
\end{equation}
To estimate it, we need the following lemma.
\begin{lemma}\label{Integral}
We have the following estimate for the integral
\begin{equation}\label{int}
\int b(\rho_2) d\rho_2 \int_{|\rho_1-\rho_2|\sim 2^{-i}}b(\rho_1) d\rho_1 \int_{|\rho_3-\rho_2|\sim 2^{-j}}b(\rho_3)b(\rho_1-\rho_2+\rho_3)
d\rho_3\lesssim 2^{-(i+j)}\|b\|_{L^4}^4.
\end{equation}
\end{lemma}

\begin{proof} We first have by H\"older's inequality
\begin{equation}\label{int1}
\begin{split}
& \int_{|\rho_3-\rho_2|\sim 2^{-j}}b(\rho_3)b(\rho_1-\rho_2+\rho_3) d\rho_3
\\
\lesssim& \left(\int_{|\rho_3-\rho_2|\sim 2^{-j}}|b(\rho_3)|^2 d\rho_3 \int_{|\rho_3-\rho_2|\sim{2^{-j}}}|b(\rho_1-\rho_2+\rho_3)|^2
d\rho_3\right)^{1/2}
\\
\lesssim& \left(\int_{|\rho_3-\rho_2|\sim 2^{-j}}|b(\rho_3)|^2 d\rho_3 \int_{|\rho|\sim{2^{-j}}}|b(\rho_1+\rho)|^2 d\rho\right)^{1/2}
\\
\lesssim& \left(\int_{|\rho_3-\rho_2|\sim 2^{-j}}|b(\rho_3)|^2 d\rho_3 \int_{|\rho-\rho_1|\sim{2^{-j}}}|b(\rho)|^2 d\rho\right)^{1/2}.
\end{split}
\end{equation}
Let $I$ be the left hand side of \eqref{int}. We estimate $I$ by \eqref{int1} and H\"older's inequality
\begin{equation*}
\begin{split}
&\int b(\rho_2)  \int_{|\rho_1-\rho_2|\sim 2^{-i}}\Big(\int_{|\rho_1-\rho|\sim 2^{-j}}|b(\rho)|^2 d\rho \Big)^{1/2}b(\rho_1) d\rho_1
\Big(\int_{|\rho_3-\rho_2|\sim 2^{-j}}|b(\rho_3)|^2  d\rho_3\Big)^{1/2} d\rho_2\\
 \lesssim &\|b\|_{L^4} \bigg\|\int_{|\rho_1-\rho_2|
\sim 2^{-i}}\Big(\int_{|\rho_1-\rho|\sim 2^{-j}}|b(\rho)|^2 d\rho \Big)^{1/2}|b(\rho_1)| d\rho_1 \bigg\|_{L^2_{\rho_2}}
\bigg\|\Big(\int_{|\rho_3-\rho_2|\sim 2^{-j}}|b(\rho_3)|^2  d\rho_3\Big)^{1/2}\bigg\|_{L^4}
\\
\lesssim&
\|b\|_{L^4}\Big\|\chi_i\ast\big((\chi_j\ast|b|^2)^\frac12|b|\big)\Big\|_{L^2}\big\|\chi_j\ast|b|^2\big\|^{1/2}_{L^2},
\end{split}
\end{equation*}
where $\chi_j=\chi_j(\rho)=\chi(2^j\rho)$ and $\chi\in C_c^\infty([\frac14,4])$. It is easy to see by the Young inequality
\begin{equation*}
\begin{split}
\left\|\chi_j\ast|b|^2\right\|^{1/2}_{L^2}\lesssim \|\chi_j\|^{1/2}_{L^1}\left\|b\right\|_{L^4}\lesssim 2^{-j/2}\left\|b\right\|_{L^4},
\end{split}
\end{equation*}
and
\begin{equation*}
\begin{split}
\left\|\chi_i\ast\left((\chi_j\ast|b|^2)^\frac12|b|\right)\right\|_{L^2}
\lesssim& \|\chi_i\|_{L^1}\left\|(\chi_j\ast|b|^2)^\frac12|b|\right\|_{L^2}\\
\lesssim& \|\chi_i\|_{L^1}\big\|\chi_j\ast|b|^2\big\|_{L^2}^\frac12\|b\|_{L^4}\\
\lesssim& 2^{-i}2^{-j/2}\left\|b\right\|^2_{L^4}.
\end{split}
\end{equation*}
Collecting the above estimates,  we obtain
\begin{equation*}
\begin{split}
I\lesssim 2^{-(i+j)}\|b\|^4_{L^4}.
\end{split}
\end{equation*}
This completes  the proof of  Lemma \ref{Integral}.
\end{proof}
Now we return to prove \eqref{eq-goal}.  Applying Lemma \ref{Integral} to \eqref{dec}, we have
\begin{equation}
\begin{split}
&\int_{\R^3}\frac{b(\rho_1)b(\rho_2)b(\rho_3)b(\rho_1-\rho_2+\rho_3)}{\left(1+R|\rho_1-\rho_2||\rho_3-\rho_2|\right)^N}d\rho_1 d\rho_2 d\rho_3\\
\lesssim&\bigg(\sum_{\{(i,j)\in \mathbb{N}^2; i+j\geq k_0\}}2^{-(i+j)}+R^{-N}\sum_{\{(i,j)\in \mathbb{N}^2; i+j\lesssim k_0\}}2^{(N-1)(i+j)}\bigg)
\|b\|^4_{L^4}
\\
\lesssim& R^{-1+\epsilon}\|b\|^4_{L^4}.
\end{split}
\end{equation}
Hence we prove \eqref{eq-goal}, and so, we finish the proof  of \eqref{q=4}.

\end{proof}

We next prove \eqref{q=6} in Proposition \ref{Local}.   We need to prove the following lemma.

\begin{lemma}\label{add-lemma4.6} Let $R\gg1$ and $f\in \mathbb{L}_1$, we have the following estimate for every $0<\epsilon\ll1$
\begin{equation}\label{R>1;6}
\begin{split}
\|(f~d\sigma)^\vee\|_{L^6_{t,x}(\R\times
A_R)}\lesssim R^{-\frac{n-1}3+\epsilon}\|g\|_{L^2_{\rho}H_\omega^{\frac{n-1}3,2}(\mathbb{S}^{n-1}) },\end{split}
\end{equation}
where $g(\xi)=f(|\xi|^2,\xi)$.
\end{lemma}

\begin{proof} It  suffices to estimate, by a scaling argument, the following quantity
\begin{equation}
\begin{split}
 \bigg\| r^{-\frac{n-2}2}\sum_{k=0}^{\infty}\sum_{\ell=1}^{d(k)}
i^{k}Y_{k,\ell}(\theta) \int_0^\infty e^{- it\rho^2}
J_{\nu(k)}( r\rho)a_{k,\ell}(\rho)\rho^{\frac
n2}\varphi(\rho)~d\rho
\bigg\|_{L^6_{t,x}(\R\times
A_R)}.
\end{split}
\end{equation}
We divide the above integral into three cases.

$\bullet$ Case 1: $k\in \Omega_1:=\{k:R\ll \nu(k)\}$.
Using \eqref{R<k} with $q=6$, we prove
\begin{equation*}
\begin{split}
& \bigg\| r^{-\frac{n-2}2}\sum_{k\in\Omega_1}\sum_{\ell=1}^{d(k)}
i^{k}Y_{k,\ell}(\theta) \int_0^\infty e^{- it\rho^2}
J_{\nu(k)}( r\rho)a_{k,\ell}(\rho)\rho^{\frac
n2}\varphi(\rho)~d\rho
\bigg\|_{L^6_{t,x}(\R\times
A_R)} \\
\lesssim& e^{-cR}
\bigg\|\Big(\sum_{k=0}^\infty\sum_{\ell=1}^{d(k)}(1+k)^{2(n-1)/3}\left|a_{k,\ell}(\rho)\right|^2\Big)^{\frac12}
\varphi(\rho)\bigg\|_{L^{2}_{\rho}}\lesssim e^{-cR} \|g\|_{L_\rho^2 H_\omega^{\frac{n-1}3,2}(\mathbb{S}^{n-1})}.
\end{split}
\end{equation*}

$\bullet$ Case 2: $k\in \Omega_2:=\{k: \nu(k)\sim R
\}$. Applying \eqref{Rk} with $q=6$ and $p=2$, we show
\begin{equation}
\begin{split}
&\bigg\| r^{-\frac{n-2}2}\sum_{k\in\Omega_2}\sum_{\ell=1}^{d(k)}
i^{k}Y_{k,\ell}(\theta) \int_0^\infty e^{- it\rho^2}
J_{\nu(k)}( r\rho)a_{k,\ell}(\rho)\rho^{\frac
n2}\varphi(\rho)~d\rho
\bigg\|_{L^6_{t,x}(\R\times
A_R)}\\
\lesssim&
R^{-(n-1)/3}\|g\|_{L_\rho^2 H_\omega^{\frac{n-1}3,2}(\mathbb{S}^{n-1})}.
\end{split}
\end{equation}

$\bullet$ Case 3: $k\in \Omega_3:=\{k:  \nu(k)\ll R\}$.  We introduce
the operator
$$T_{\nu}(a)(t,r)=\chi\Big(\frac r R\Big)\int_0^\infty
e^{- it\rho^2} h_{\nu}(
r\rho)a(\rho)\rho^{\frac n2}\varphi(\rho) d\rho$$
where $|h_\nu(r)|\leq C/r$ and the operator
$$H_{\nu}(a)(t,r)=\chi\Big(\frac r R\Big)\int_0^\infty
e^{- it\rho^2} I_{\nu}( r\rho)a(\rho)\rho^{\frac n2}\varphi(\rho)
d\rho,$$ where $\nu=\nu(k)=k+(n-2)/2$. Since
$$J_\nu(r)=I_{\nu}(r)+\bar{I}_{\nu}(r)+h_\nu(r),$$
 our aim here is to estimate
\begin{equation*}
\begin{split}
& \bigg\| r^{-\frac{n-2}2}\sum_{k\in\Omega_3}\sum_{\ell=1}^{d(k)}
i^{k}Y_{k,\ell}(\theta) \int_0^\infty e^{- it\rho^2}
J_{\nu(k)}( r\rho)a_{k,\ell}(\rho)\rho^{\frac
n2}\varphi(\rho)~d\rho
\bigg\|_{L^6_{t,x}(\R\times
A_R)}\\
\lesssim& R^{-\frac{n-1}3+\frac12}\bigg(\sum_{k\in\Omega_3}\sum_{\ell=1}^{d(k)}(1+k)^{2(n-1)/3}
\Big( \left\|T_{\nu(k)}(a_{k,\ell})(t,r)\right\|_{L^6_t(\R;L^6_{r}(S_R))}^{2}+\\
&\qquad\qquad \left\|H_{\nu(k)}(a_{k,\ell})(t,r)
\right\|_{L^6_t(\R;L^6_{r}(S_R))}^{2}\Big)\bigg)^{1/2}.
\end{split}
\end{equation*}
By making use of  \eqref{T} with $q=6$, we have
$$\|T_{\nu}(a)(t,r)\|_{L^6_{t,r}}\leq R^{-5/6}\|a\varphi\|_{L^{6/5}}.$$
This implies that
\begin{equation}\label{6.1}
\begin{split}
&\bigg(\sum_{k\in\Omega_3}\sum_{\ell=1}^{d(k)}(1+k)^{2(n-1)/3}
 \left\|T_{\nu(k)}(a_{k,\ell})(t,r)\right\|_{L^6_t(\R;L^6_{r}(S_R))}^{2}\bigg)^{1/2}\\
 \lesssim&
 R^{-5/6}\bigg\|\Big(\sum_{k\in\Omega_3}\sum_{\ell=1}^{d(k)}(1+k)^{2(n-1)/3}
\left|a_{k,\ell}(\rho)\right|^2\Big)^{1/2}\varphi \bigg\|_{L^{6/5}}\\
\lesssim&  R^{-5/6}\|g\|_{L_\rho^2 H_\omega^{\frac{n-1}3,2}(\mathbb{S}^{n-1})}.
\end{split}
\end{equation}

 On the other hand, by \eqref{b3},  one has $|I_\nu(r)|\lesssim r^{-1/2}$ when $k\in\Omega_3$.
 Consider the operator
$$H_{\nu}(a)(t,r)=\chi\Big(\frac r R\Big)\int_0^\infty
e^{- it\rho^2} I_{\nu}(
r\rho)a(\rho)\rho^{\frac n2}\varphi(\rho) d\rho,$$
where $\nu=\nu(k)=k+(n-2)/2$ with $k\in\Omega_3$.

On the one hand, it is easy to see
 \begin{equation*}
\left\|H_{\nu}(a)(t,r)\right\|_{L^\infty_{t,r}(\R\times\R^n)}\lesssim R^{-1/2}\|a \varphi\|_{L^1}.
\end{equation*}
On the other hand, we have the claim that for any $\epsilon>0$
\begin{equation}\label{claim}
\left\|H_{\nu}(a)(t,r)\right\|_{L^4_{t,r}(\R\times\R)}\lesssim
R^{-1/2+\epsilon}\|a \varphi\|_{L^4_\rho}.
\end{equation}
We postpone the proof of this claim to the end of this section.
Hence, by the interpolation of the above two estimates, for any $\epsilon>0$, we obtain that
\begin{equation*}
\left\|H_{\nu}(a)(t,r)\right\|_{L^6_{t,r}(\R\times\R^n)}\lesssim R^{-1/2+\epsilon}\|a \varphi\|_{L^2}.
\end{equation*}
This shows
\begin{equation}\label{6.2}
\begin{split}
&\bigg(\sum_{k\in\Omega_3}\sum_{\ell=1}^{d(k)}(1+k)^{2(n-1)/3}
 \left\|H_{\nu(k)}(a_{k,\ell})(t,r)\right\|_{L^6_t(\R;L^6_{r}(S_R))}^{2}\bigg)^{1/2}\\
 \lesssim&
 R^{-1/2+\epsilon}\Big(\sum_{k\in\Omega_3}\sum_{\ell=1}^{d(k)}(1+k)^{2(n-1)/3}
\left\|a_{k,\ell}(\rho)\varphi(\rho)\right\|^2_{L^{2}}\Big)^{1/2}\\
\lesssim&
 R^{-1/2+\epsilon}\|g\|_{L^2_{\rho}H_\omega^{\frac{n-1}3,2}(\mathbb{S}^{n-1}) }.
\end{split}
\end{equation}
Collecting \eqref{6.1} and \eqref{6.2} yields
\begin{equation*}
\begin{split}
& \bigg\| r^{-\frac{n-2}2}\sum_{k\in\Omega_3}\sum_{\ell=1}^{d(k)}
i^{k}Y_{k,\ell}(\theta) \int_0^\infty e^{- it\rho^2}
J_{\nu(k)}( r\rho)a_{k,\ell}(\rho)\rho^{\frac
n2}\varphi(\rho)~d\rho
\bigg\|_{L^6_{t,x}(\R\times
A_R)}\\
\lesssim& R^{-\frac{n-1}3+\epsilon}\|g\|_{L^2_{\rho}H_\omega^{\frac{n-1}3,2}(\mathbb{S}^{n-1}) }.
\end{split}
\end{equation*}
This implies \eqref{R>1;6}, which completes the proof of Lemma \ref{add-lemma4.6}.

\end{proof}
\begin{proof}[The proof of claim \eqref{claim}] The same argument in the proof the \eqref{Iv} shows the claim \eqref{claim}.
Recall the kernel \eqref{kernel}, it is enough to estimate the integral
\begin{equation*}
\begin{split}
&\left\|H_{\nu}(a)(t,r)\right\|^4_{L^4_{t,r}(\R\times\R^n)}\\
=&\int_{\R^4} \int_\R e^{- it(\rho_1^2-\rho_2^2+\rho_3^2-\rho_4^2)}
K(R,\nu;\rho_1,\rho_2,\rho_3,\rho_4) a(\rho_1)
a(\rho_2)a(\rho_3)a(\rho_4)\\
&\qquad\qquad \beta(\rho_1)\beta(\rho_2)\beta(\rho_3)\beta(\rho_4) dt d\rho_1d\rho_2d\rho_3d\rho_4,
\end{split}
\end{equation*}where $\beta(\rho)=\rho^{\frac n2}\varphi(\rho)$.
 For $b(\rho)=2a(\sqrt{\rho})\beta({\sqrt{\rho})}/\sqrt{\rho}$, therefore we obtain
\begin{equation*}
\begin{split}
&\left\|H_{\nu}(a)(t,r)\right\|^4_{L^4_{t,r}(\R\times\R^n)}\\
&=\int_{\R^4} \delta(\rho_1-\rho_2+\rho_3-\rho_4) K(R,\nu;
\sqrt{\rho_1},\sqrt{\rho_2},\sqrt{\rho_3},\sqrt{\rho_4})  b(\rho_1)
b(\rho_2)b(\rho_3)b(\rho_4)d\rho_1d\rho_2d\rho_3d\rho_4\\
=&\int_{\R^3}
K(R,\nu;
\sqrt{\rho_1},\sqrt{\rho_2},\sqrt{\rho_3},\sqrt{\rho_1-\rho_2+\rho_3})
b(\rho_1)
b(\rho_2)b(\rho_3)b(\rho_1-\rho_2+\rho_3)d\rho_1d\rho_2d\rho_3\\
\lesssim& R^{-2+\epsilon}\|b\|^4_{L^4}\lesssim R^{-2+\epsilon}\|a\varphi\|^4_{L^4} .
\end{split}
\end{equation*}
where we use the kernel estimate \eqref{decay} and \eqref{eq-goal} in the first inequality.
\end{proof}

\section{Local smoothing estimate}

K. M. Rogers \cite{Rogers}  developed an argument  showing that a restriction estimate implies a local smoothing estimate under some suitable conditions.
For the sake of convenience, we closely follow this  argument to prove Corollary \ref{cor}.
In fact, by making use of the standard Littlewood-Paley
argument, it can be reduced to prove the claim
\begin{equation}\label{smoothing}
\begin{split}
\|e^{it\Delta}(1-\Delta_\theta)^{-s/2}u_0\|_{L^q_{t,x}([0,1]\times\R^n)}
\lesssim
N^{(2n(1/2-1/q)-2/q)_+}\left\|u_0\right\|_{L^q_{x}},\quad\forall~
N\gg1
\end{split}
\end{equation}
where
 $$\mathrm{supp}~
\mathcal{F}({(1-\Delta_\theta)^{-s/2}u_0})\subset\{\xi:|\xi|\leq
N\}.$$
Here we denote by $\mathcal{F}$  the Fourier transform. We also use the notation $\hat h$ to express the Fourier transform of $h$.  Let $h=(1-\Delta_\theta)^{-s/2}u_0$. Denote by $P_N$ the
Littlewood-Paley projector, i.e.
$$P_Nh=\mathcal{F}^{-1}\Big(\chi\Big(\frac{|\xi|}N\Big)\hat
h\Big),\quad\; \chi\in \C_c^\infty([1/2,1]).$$
 By the
Littlewood-Paley theory
 and the claim \eqref{smoothing}, one has for
$\alpha>2n(1/2-1/q)-2/q$
\begin{equation*}
\begin{split}
\|e^{it\Delta}h\|^2_{L^q_{t,x}([0,1]\times\R^n)}\lesssim&\|e^{it\Delta}P_{\lesssim1}h\|^2_{L^q_{t,x}([0,1]\times\R^n)}+
\sum_{N\gg1}\left\|e^{it\Delta}P_Nh\right\|^2_{L^q_{t,x}([0,1]\times\R^n)}
\\
\lesssim&\|u_0\|^2_{L_x^q(\R^n)}+\sum_{N\gg1}N^{2[2n(1/2-1/q)-2/q]+}\big\|P_Nu_0\big\|^2_{L^q_{x}}\\
\lesssim& \|u_0\|^2_{L_x^q(\R^n)}+
\bigg\|\Big(\sum_{N\gg1}N^{q\alpha}
\left|P_Nu_0\right|^q\Big)^{1/q}\bigg\|^2 _{L^q_{x}}
\\
\lesssim& \|u_0\|^2_{L_x^q(\R^n)}+\bigg\|\Big(\sum_{N\gg1}N^{2\alpha}\left|P_Nu_0\right|^2\Big)^{1/2}\bigg\|^2_{L^q_{x}}\\
\simeq&\|u_0\|^2_{W^{\alpha,q}_{x}(\R^n)}.
\end{split}
\end{equation*}
Here we use H\"older's inequality for the third inequality, Sobolev imbedding for the fourth one.
Hence we have
\begin{equation*}
\begin{split}
\|e^{it\Delta}u_0\|_{L^q_{t,x}([0,1]\times\R^n)}\lesssim
\|(1-\Delta_\theta)^{s/2}u_0\|_{W^{\alpha,q}_{x}(\R^n)}.
\end{split}
\end{equation*}
Now we are left to prove claim \eqref{smoothing}. Assume $\mathrm{supp}~\hat {f} \subset[0, 1]$. Note that
\begin{equation*}
e^{it\Delta}f=\frac{1}{(it)^{n/2}}\int_{\R^n}
e^{i|x-y|^2/t}f(y)dy,\quad \forall~t\in\R\backslash\{0\}.
\end{equation*}
On the other hand, we have for $t\neq0$
\begin{equation*}
\begin{split}
e^{it\Delta}f=&\int_{\R^n}e^{i(t|\xi|^2+x\cdot\xi)}\hat{f}(\xi)d\xi=e^{-\frac{i|x|^2}{4t}}\int_{\R^n}e^{it|\xi+\frac{x}{2t}|^2}\hat{f}(\xi)d\xi\\
=&\frac{1}{(it)^{n/2}}e^{-\frac{i|x|^2}{4t}}\left(e^{i\frac{\Delta}t}\hat{f}\right)\left(-\frac{x}{2t}\right).
\end{split}
\end{equation*}
So we have for every dyadic number $N$
\begin{equation*}
\begin{split}
\|e^{it\Delta}f\|_{L^q_{t,x}(|t|\sim N^2; |x|\lesssim N^2)}\lesssim& N^{-n}\left\|\left(e^{i\frac{\Delta}t}\hat{f}\right)
\left(-\frac{\bullet}{2t}\right)\right\|_{L^q_{t,x}(|t|\sim N^2; |x|\lesssim N^2)}\\
\lesssim& N^{-n+\frac{2n+4}{q}}\left\|e^{it\Delta}\hat{f}\right\|_{L^q_{t,x}(|t|\sim N^{-2}; |x|\lesssim 1)}.
\end{split}
\end{equation*}
By making use of Theorem \ref{thm}, we obtain for $q>2(n+1)/n$ and $\frac{n+2}q=\frac n{p'}$
\begin{equation}
\begin{split}
\left\|e^{it\Delta}\hat{f}\right\|_{L^q_{t,x}(|t|\sim N^{-2}; |x|\lesssim 1)}\lesssim \|f\|_{L^p_{\mu(r)}(\R^+;H^{s,p}_\theta(\mathbb{S}^{n-1}))}.
\end{split}
\end{equation}
This yields
\begin{equation*}
\begin{split}
\|e^{it\Delta}f\|_{L^q_{t,x}(|t|\sim N^2; |x|\lesssim N^2)}\lesssim N^{-n+\frac{2n+4}{q}}\|f\|_{L^p_{\mu(r)}(\R^+;H^{s,p}_\theta(\mathbb{S}^{n-1}))}.
\end{split}
\end{equation*}
This  implies that
\begin{equation}
\begin{split}
\|e^{it\Delta}(1-\Delta_\theta)^{-s/2}f\|_{L^q_{t,x}(|t|\sim N^2; |x|\lesssim N^2)}\lesssim N^{-n+\frac{2n+4}{q}}\|f\|_{L^p_{x}}.
\end{split}
\end{equation}
For the sake of convenience, we recall \cite [Lemma 8] {Rogers}
\begin{lemma}Let $q\geq p_1\geq p_0$, $r\geq1$ and $I\subset[0,R^2]$. If one has
$$\|e^{it\Delta}f\|_{L^q_x(B_{R^2};L^r_t(I))}\leq CR^s\|f\|_{L^{p_0}(\R^n)}$$

where $R\gg1$, and $f$ is frequency supported in unite ball $\mathbb{B}^n$. Then for all $\epsilon>0$
$$\|e^{it\Delta}f\|_{L^q_x(\R^n;L^r_t(I))}\leq C_\epsilon R^{s+2n(\frac1{p_0}-\frac1{p_1})+\epsilon}\|f\|_{L^{p_1}(\R^n)}.$$

\end{lemma}

Since $q>p$ when $q>2(n+1)/n$, for any $0<\epsilon\ll 1$,
we have by this lemma
\begin{equation*}
\begin{split}
&\|e^{it\Delta}(1-\Delta_\theta)^{-s/2}f\|_{L^q_{t,x}(|t|\sim N^2; x\in\R^n)}\\
\lesssim& N^{-n+\frac{2n+4}{q}+2n(\frac1p-\frac1q)
+\epsilon}\|f\|_{L^q_{x}}\\
\lesssim& N^{n(1-\frac2q)+\epsilon}\|f\|_{L^q_{x}}.
\end{split}
\end{equation*}
Using the scaling argument, if
$${\rm supp}\widehat {f_{k,N}}\subset B_{2^{k/2}N}:=\big\{\xi: |\xi|\in [0, 2^{k/2}N]\big\},\quad\; \forall\; k\geq 0,$$
then
\begin{equation}
\begin{split}
\|e^{it\Delta} (1-\Delta_\theta)^{-\frac{s}2}f_{k,N}\|_{L^q_{t,x}([2^{-k},2^{-k+1}]\times\R^n)}
\lesssim N^{n(1-\frac2q)+\epsilon}(2^{\frac{k}2}N)^{-\frac2q}\left\|f_{k,N}\right\|_{L^q_{x}}.
\end{split}
\end{equation}
Since
$${\rm supp}\hat{h} \subset \{\xi: |\xi|\in [N/2, N]\}\subset B_{2^{k/2}N},\quad\; \forall k\geq2,$$
 we replace $(1-\Delta_\theta)^{-s/2}f_{k,N}$ by $h$ to obtain
\begin{equation}
\begin{split}
\|e^{it\Delta}h\|_{L^q_{t,x}([0,1]\times\R^n)}=&\bigg(\sum_{k\geq0}\|e^{it\Delta} (1-\Delta_\theta)^{-s/2}u_0\|^q_{L^q_{t,x}
([2^{-k},2^{-k+1}]\times\R^n)}\bigg)^{1/q}
\\
\lesssim& \Big(\sum_{k\geq
0}2^{-k}\Big)^{1/q}N^{(2n(1/2-1/q)-2/q)_+}\left\|u_0\right\|_{L^q_{x}}.
\end{split}
\end{equation}
 This proves inequality \eqref{smoothing}.

\begin{center}

\end{center}
\end{document}